\documentclass{amsart}
\usepackage{amssymb,mathrsfs,MnSymbol}

\def\bC {\mathbf{C}}

\def\bR {\mathbf{R}}

\def\fH {\mathfrak{H}}
\def\fS {\mathfrak{S}}

\def\rI {\mathrm{I}}

\def\cC {\mathcal{C}}
\def\cD {\mathcal{D}}

\def\cF {\mathcal{F}}

\def\cL {\mathcal{L}}

\def\cQ {\mathcal{Q}}

\def\cS {\mathcal{S}}

\def\cV {\mathcal{V}}

\def\de {{\delta}}
\def\eps {{\epsilon}}

\def\l {{\lambda}}
\def\L {{\Lambda}}
\def\si {{\sigma}}

\def\d {{\partial}}
\def\dd{\,\mathrm{d}}
\def\grad {{\nabla}}
\def\Dlt {{\Delta}}

\def\rstr {{\big |}}

\def\la {\langle}
\def\ra {\rangle}

\newcommand{\Tr}{\operatorname{trace}}

\newcommand{\Lip}{\operatorname{Lip}}

\newcommand{\Op}{\operatorname{Op}}
\newcommand{\MKd}{\operatorname{dist_{MK,2}}}
\newcommand{\MK}{\operatorname{MK}}

\newcommand{\ba}{\begin{aligned}}
\newcommand{\ea}{\end{aligned}}

\newcommand{\be}{\begin{equation}}
\newcommand{\ee}{\end{equation}}

\newcommand{\lb}{\label}

\newtheorem{Thm}{Theorem}[section]

\newtheorem{Prop}[Thm]{Proposition}
\newtheorem{Cor}[Thm]{Corollary}
\newtheorem{Lem}[Thm]{Lemma}
\newtheorem{Def}[Thm]{Definition}

\newcommand{\cq}{q}
\newcommand{\gp}{p}

\begin{document}

\title[Wave Packets and Monge-Kantorovich Distance]{Wave Packets and\\ the Quadratic Monge-Kantorovich Distance\\ in Quantum Mechanics\\ }

\author[F. Golse]{Fran\c cois Golse}
\address[F.G.]{CMLS, \'Ecole polytechnique, CNRS, Universit\'e Paris-Saclay , 91128 Palaiseau Cedex, France}
\email{francois.golse@polytechnique.edu}

\author[T. Paul]{Thierry Paul}
\address[T.P..]{CMLS, \'Ecole polytechnique, CNRS, Universit\'e Paris-Saclay , 91128 Palaiseau Cedex, France}
\email{thierry.paul@polytechnique.edu}

\begin{abstract}
In this paper, we extend the upper and lower bounds for the ``pseudo-distance'' on quantum densities analogous to the quadratic Monge-Kantorovich(-Vasershtein) distance introduced in [F. Golse, C. Mouhot, T. Paul, Commun. Math. 
Phys. \textbf{343} (2016) 165--205] to  positive quantizations defined in terms of the family of phase space translates of a density operator, not necessarily of rank $1$ as in the case of the T\"oplitz quantization. As a corollary, we prove 
that the uniform as $\hbar\to 0$ convergence rate for the mean-field limit of the $N$-particle Heisenberg equation holds for a much wider class of initial data than in [F. Golse, C. Mouhot, T. Paul, loc. cit.]. We also discuss the relevance 
of the pseudo-distance compared to the Schatten norms for the purpose of  metrizing the set of quantum density operators in the semiclassical regime.
\end{abstract}




\maketitle
\tableofcontents


\section{Generalized Husimi Transform and Positive Quantization}


Let $\fH:=L^2(\bR^d)$; a density operator on $\fH$ is a bounded operator $R$ on $\fH$ such that
$$
R=R^*\ge 0\quad\hbox{ and }\Tr(R)=1\,.
$$
We denote by $\cD(\fH)$ the set of density operators on $\fH$, and set
$$
\cD^2(\fH):=\{R\in\cD(\fH)\hbox{ s.t. }\Tr(R^{1/2}|x|^2R^{1/2})+\Tr(R^{1/2}(-\Dlt_x)R^{1/2})<\infty\}\,.
$$
For all $\cq,\gp\in\bR^d$ and $\l>0$, and for all $\psi\in\fH$, we set
$$
T_{\cq,\gp}\psi(x)=\psi(x-\cq)e^{i\gp\cdot(x-\cq/2)}\,,\quad\hbox{Êand }\quad S_\l\psi(x)=\l^{-d/4}\psi(x/\sqrt{\l})\,.
$$
One has obviously
$$
T_{\cq+\cq',\gp+\gp'}=e^{-i(\gp\cdot\cq'-\gp'\cdot\cq)/2}\,T_{\cq,\gp}T_{\cq'\gp'}\quad\hbox{ and }S_{\l\l'}=S_\l S_{\l'}
$$
for all $\cq,\cq',\gp,\gp'\in\bR^d$ and $\l,\l'>0$, and 
$$
T^*_{\cq,\gp}=T_{-\cq,-\gp}=T^{-1}_{\cq,\gp}\quad\hbox{ and }\quad S^*_\l=S_{1/\l}=S^{-1}_\l\,,
$$
so that $T_{\cq,\gp}$ and $S_\l$ are unitary operators on $\fH$. 

We set
$$
R^\l_{\cq,\gp}:=T_{\cq,\gp/\l}S_\l RS^*_\l T^*_{\cq,\gp/\l}\qquad\hbox{ for each } R\in \cD(\fH),\cq,\gp\in\bR^d\,,\,\,\l>0\,.
$$
The family $R^\l_{\cq,\gp}$ is, for each $\l>0$, a resolution of the identity, i.e.
\be\lb{ResolID}
\frac1{(2\pi\l)^d}\int_{\bR^d\times\bR^d}R^\l_{\cq,\gp}\dd \cq\dd \gp=\rI_\fH\,,
\ee
the integral on the left hand side  being understood in the weak sense, i.e., for each $\phi,\psi\in\fH$, the function $(\cq,\gp)\mapsto\la\phi|R_{\cq,\gp}|\psi\ra$ belongs to $L^1(\bR^d\times\bR^d)$ and
\be\lb{ResolID2}
\frac1{(2\pi\l)^d}\int_{\bR^d\times\bR^d}\la\phi|R^\l_{\cq,\gp}|\psi\ra\dd \cq\dd \gp=\la\phi|\psi\ra\,.
\ee
Indeed\footnote{Although we have given an explicit proof of \eqref{ResolID}, one could also use the following argument. Since the family of Weyl operators $e^{i\theta}T_{q,p}$ with $\theta\in S^1$ and $(q,p)\in T^*\bR^d,$ defines 
an irreducible representation of the Weyl-Heisenberg group, \eqref{ResolID2} can be recovered from the so-called orthogonality relations of square integrable group representations (see \cite{GMP}, Theorem 3.1) applied to each 
term of the spectral decomposition of the Hilbert-Schmidt operator $S_\l RS^*_\l$.} let $r(x,x')$ be the integral kernel of $R$. The integral kernel of the left hand side of \eqref{ResolID} is
$$
\ba
\frac1{(2\pi\l)^d}\int_{\bR^d\times\bR^d}\l^{-d/2}r(\tfrac {x-q}{\sqrt\l},\tfrac{x'-q}{\sqrt\l})e^{ip(x-x')/\l}\dd \cq\dd \gp
\\
=\delta(x-x')\int_{\bR^d\times\bR^d}\l^{-d/2}r(\tfrac {x-q}{\sqrt\l},\tfrac{x'-q}{\sqrt\l})\dd \cq
\\
=\delta(x-x')\Tr(R)=\delta(x-x')&\,.
\ea
$$

\medskip
The following definition generalizes the standard T\"oplitz quantization. 

\begin{Def}\lb{D-Opmu}
Let $R\in\cD(\fH)$. For each positive Borel measure $\mu$ on $\bR^d\times\bR^d$ and each $\l>0$, we denote by $\Op^R_\l[\mu]$ the (possibly unbounded) nonnegative self-adjoint operator on $L^2(\bR^d)$ given by
$$
\Op^R_\l[\mu]:=\frac1{(2\pi\l)^d}\int_{\bR^d\times\bR^d}R^\l_{q,p}\,\mu(\dd p\dd q)\,.
$$
\end{Def}

(Denoting by $\cV_R\subset\fH$ the closed linear subspace of functions $\phi\equiv\phi(x)$ such that $(p,q)\mapsto\la\phi|R^\l_{q,p}|\phi\ra$ belongs to $L^1(\bR^d\times\bR^d,\mu)$, the formula above defines $\Op^R_\l[\mu]$ as a 
bounded linear operator from $\cV_R$ to its topological dual $\cV'_R$.)

Notice that $\Op^R_\l[\mu]$ can be expressed as a sum of standard ``rank one" T\"oplitz operators by using the spectral decomposition of the Hilbert-Schmidt operator $R$
\smallskip

\noindent
\textbf{Example.}  Let $a\in H^1(\bR^d)$ satisfy
$$
\int_{\bR^d}|a(y)|^2\dd y=1\,,\quad\int_{\bR^d}|y|^2|a(y)|^2\dd y<\infty\,.
$$
Then, the orthogonal projection on $\bC a$ belongs to $\cD^2(\fH)$.

Henceforth we set
\be\label{defcoh}
|q,p,\l,a\ra:=T_{q,p}S_{\l}a\,,\qquad p,q\in\bR^d\,,\,\,\l>0\,,
\ee
and use Dirac's notation involving bras and kets (see chapter II.B in \cite{CohTan}).

For instance, one can choose $a$ to be a Gaussian:
\be\lb{Gauss}
a(x):=\pi^{-d/4}e^{-|x|^2/2}\,,
\ee
in which case $|p,q,\hbar,a\ra$ (where $\hbar$ is the Planck constant) designates the Schr\"odinger coherent state (\cite{schro}, Problem 3 in \S 23 of \cite{LL3}). 

\bigskip
Next we recall the notion of Wigner transform at scale $\l$ of a Hilbert-Schmidt operator $K$ on $L^2(\bR^d)$, with integral kernel $k\in L^2(\bR^d\times\bR^d)$ (see formula (51) in \cite{LionsPaul}):
\be\lb{DefW}
W_\l[K](x,\xi):=(2\pi)^{-d}\cF_{y\to\xi}\left(k(x+\tfrac12\l y,x-\tfrac12\l y)\right)\,.
\ee
(The notation $\cF_{y\to\xi}$ designates the partial Fourier transform defined by the formula
$$
\cF_{y\to\xi}(\phi(x,y)):=\int_{\bR^d}\phi(x,y)e^{-i\xi\cdot y}\dd y\quad\hbox{ for all }\phi\in\cS(\bR^d\times\bR^d)\,,
$$
and extended by duality to $\cS'(\bR^d\times\bR^d)$.) 

\smallskip
The Wigner transform satisfies the following elementary properties. 

\begin{Prop} For all Hilbert-Schmidt operators $K,L$ on $L^2(\bR^d)$ and all $\l>0$,
\be\lb{AdjW}
W_\l[K^*]=\overline{W_\l[K]}\,,
\ee
and
\be\lb{TrQ*R}
\Tr(K^*L)=(2\pi\l)^d\int_{\bR^d\times\bR^d}\overline{W_\l[K](x,\xi)}W_\l[L](x,\xi)\dd x\dd\xi\,.
\ee
For each $p,q\in\bR^d$, one has
\be\lb{TranslW}
W_\l[\,T_{q,p/\l}KT^*_{q,p/\l}\,](x,\xi)=W_\l[\,K\,](x-q,\xi-p)\,,\quad\hbox{ for a.e. }x,\xi\in\bR^d\,.
\ee
For each Borel probability measure $\mu$ on $\bR^d\times\bR^d$, one has
\be\lb{WOp}
W_\l[\,\Op^R_\l[(2\pi\l)^d\mu]\,]=\mu\star W_\l[\,R\,]\,,
\ee
and\footnote{For each complex-valued function $f$ defined a.e. on $\bR^n$, we denote $f^*(x):=\overline{f(-x)}$.}
\be\lb{pq00}
W_\l[\,R\,]\star W_\l[\,R\,]^*(q,p)=\frac{\Tr(|(R^\l)^{1/2}T_{q,p}(R^\l)^{1/2}|^2)}{(2\pi\l)^d}\ge 0\,.
\ee
\end{Prop}

\begin{proof}
If $k\equiv k(X,Y)$ is the integral kernel of $K$, the integral kernel of $K^*$ is $\overline{k(Y,X)}$, and this implies (\ref{AdjW}). Likewise, the integral kernel of $T_{q,p/\l}KT^*_{q,p/\l}$ is
$$
k(x-q,y-q)e^{ip\cdot(x-y)/\l}\,,
$$
and this implies formula (\ref{TranslW}). Formula (\ref{WOp}) follows from formula (\ref{TranslW}) and Fubini's theorem. To prove (\ref{TrQ*R}), denote by $k$ and $l$ the integral kernels of $K$ and $L$ respectively, write
$$
\ba
\Tr(K^*L)=\int_{\bR^d\times\bR^d}\overline{k(Y,X)}l(Y,X)\dd X\dd Y
\\
=\l^d\int_{\bR^d}\left(\int_{\bR^d}\overline{k(x-\tfrac12\l y,x+\tfrac12\l y)}l(x-\tfrac12\l y,x+\tfrac12\l y)\dd y\right)\dd x&\,,
\ea
$$
and apply Plancherel's theorem to the inner integral on the right hand side. Finally, formula (\ref{pq00}) follows from the identities (\ref{TrQ*R}) and (\ref{TranslW}).
\end{proof}

\smallskip
Along with the generalization of the standard T\"oplitz quantization given in Definition \ref{D-Opmu}, we define a notion of generalized Husimi transform. We refer to \cite{LionsPaul} for the theory of the usual Husimi transform, namely in 
the case where $R=|a\ra\la a|$, with $a$ chosen to be the Gaussian state (\ref{Gauss}).

\begin{Def}\lb{D-GenHusimi}
Let $R\in\cD(\fH)$, and let $K$ be a Hilbert-Schmidt operator on $L^2(\bR^d)$. Its generalized Husimi transform is 
$$
\tilde W^R_\l[K]:=W_\l[K]\star W_\l[\,R\,]^*\,.
$$
\end{Def}

\smallskip
In the case where $a$ is the Gaussian profile (\ref{Gauss}), an elementary computation shows that
$$
W_\l[\,|a\ra\la a|\,](x,\xi)=(\pi\l)^{-d}e^{-(|x|^2+|\xi|^2)/\l}\,,
$$
so that the definition of $\tilde W^R_\l[K]$ given above with $R=|a\ra\la a|$ coincides with formula (52) in \cite{LionsPaul}.

\smallskip
The following properties of this generalized Husimi transform are very similar to those already known in the Gaussian case (see \cite{LionsPaul}).

\begin{Prop}
Let $K$ be a Hilbert-Schmidt operator on $L^2(\bR^d)$. Then, for all $\l>0$
\be\lb{HusiPosi}
K=K^*\ge 0\quad\Rightarrow\quad\tilde W^R_\l[K]\ge 0\hbox{ on }\bR^d\times\bR^d\,.
\ee
In particular, for each Borel probability measure $\mu$ on $\bR^d\times\bR^d$, one has
\be\lb{HusiTopli}
\tilde W^R_\l[\,\Op^R_\l[(2\pi\l)^d\mu]\,](q,p)=\int_{\bR^d\times\bR^d}\frac{\Tr(|(R^\l)^{1/2}T_{q-q',p-p'}(R^\l)^{1/2}|^2)}{(2\pi\l)^d}\mu(\dd p'\dd q')\,.
\ee
\end{Prop}

\begin{proof}
By (\ref{AdjW}), (\ref{TrQ*R}) and (\ref{TranslW}), one has
$$
\tilde W^R_\l[K](q,p)=\int_{\bR^d\times\bR^d}W_\l[K](x,\xi)W_\l[\,R_{q,p}\,]^*(x,\xi)\dd x\dd\xi=\frac{\Tr(R^\l_{q,p}K)}{(2\pi\l)^d}\,.
$$

Next, one has
$$
\Tr(R^\l_{q,p}K)=\Tr((R^\l_{q,p})^{1/2}K(R^\l_{q,p})^{1/2})\ge 0\,,
$$
Indeed $K=K^*\ge 0$ and
$$
R^\l_{q,p}=T_{q,p/\l}S_\l RS^*_\l T^*_{q,p/\l}=(R^\l_{q,p})^*\ge 0\,,\quad\hbox{ since }R=R^*\ge 0\,.
$$
This observation proves the inequality (\ref{HusiPosi}) and generalizes formula (42) in \cite{LionsPaul}. 

The identity (\ref{HusiTopli}) follows from Definition \ref{D-GenHusimi} with formulas (\ref{WOp}) and (\ref{pq00}), after observing that
$$
\ba
\Tr(R^\l R^\l_{q-q',p-p'})=&\Tr(R^\l T_{q-q',(p-p')/\l}R^\l T^*_{q-q',(p-p')/\l})
\\
=&\Tr(T_{q',p'/\l}R^\l T_{q-q',(p-p')/\l}R^\l T^*_{q,p/\l})
\\
=&\Tr(T_{q',p'/\l}R^\l T^*_{q',p'/\l}T_{q,p/\l}R^\l T^*_{q,p/\l})=\Tr(R^\l_{q,p}R^\l_{q',p'})
\ea
$$
for all $p,p',q,q'\in\bR^d$.
\end{proof}


\section{Monge-Kantorovich Distance and Positive Quantization:\\ an Upper Bound}


We recall the following notion of ``pseudo-distance''\footnote{There exists a well-defined notion of pseudometric space. We do not claim that the functional $MK_{\hbar}$ defined below is a pseudometric; we nevertheless call $MK_{\hbar}$
a pseudo-distance for want of a better terminology.} between density operators on $\fH=L^2(\bR^d)$ introduced in Definition 2.2 of \cite{FGMouPaul}.

For $K,K'\in\cD(\fH)$, a \textit{coupling} of $K,K'$ is an element $Q\in\cD(\fH\otimes\fH)$ such that, for all bounded operators $A,A'$ on $\fH$
$$
\Tr_{\fH\otimes\fH}(Q(A\otimes\rI+\rI\otimes A'))=\Tr_\fH(KA)+\Tr_\fH(K'A')\,.
$$
(See Definition 2.1 in \cite{FGMouPaul}.) The set of couplings of $K,K'$ is denoted $\cC(K,K')$. Obviously $K\otimes K'\in\cC(K,K')$, so that $\cC(K,K')\not=\varnothing$. 

For each pair $K,K'\in\cD(\fH)$ and each $\l>0$, set
$$
\MK_\l(K,K'):=\inf_{Q\in\cC(K,K')}\sqrt{\,\Tr_{\fH\otimes\fH}(Q^{1/2}C_\l(x,x',\grad_x,\grad_{x'})Q^{1/2})\,}\in[0,+\infty]\,,
$$
where 
$$
C_\l(x,x',\grad_x,\grad_{x'}):=\sum_{j=1}^d\left((x_j-x'_j)^2-\l^2(\d_{x_j}-\d_{x'_j})^2\right)\,.
$$
This definition is formally analogous to the definition of the Monge-Kantorovich, or Vasershtein distance of exponent $2$ (see Theorem 7.3 in chapter 7 of \cite{VillaniTOT}). In the language of optimal transportation, the differential operator
$C_\l$ above is analogous to the notion of cost function (see chapter 1 in \cite{VillaniTOT}).

\smallskip
We begin with an elementary observation, which is the analogue of Proposition 2.1 in \cite{VillaniTOT}.

\begin{Lem}\lb{L-Inf=Min}
For each pair $K,K'\in\cD^2(\fH)$ and each $\l>0$, there exists $Q\in\cC(K,K')$ such that
$$
\MK_\l(K,K')^2=\Tr_{\fH\otimes\fH}(Q_n^{1/2}C_\l(x,x',\grad_x,\grad_{x'})Q_n^{1/2})\,.
$$
\end{Lem}

\begin{proof}
Let $Q_n\in\cC(K,K')$ be a minimizing sequence, i.e.
$$
\Tr_{\fH\otimes\fH}(Q_n^{1/2}C_\l(x,x',\grad_x,\grad_{x'})Q_n^{1/2})\to\MK_\l(K,K')^2
$$
as $n\to\infty$. Since $Q_n\in\cC(K,K')$, one has
$$
\Tr_{\fH\otimes\fH}(Q_n^{1/2}(H\otimes\rI_\fH+\rI_\fH\otimes H)Q_n^{1/2})=\Tr_{\fH}(HK)+\Tr_{\fH}(HK')<\infty
$$
for all $n\ge 1$, where
$$
H:=|x|^2-\Dlt_x\,.
$$
(That $\Tr_{\fH}(HK)+\Tr_{\fH}(HK')<\infty$ follows from the fact that $K,K'\in\cD^2(\fH)$.) By Proposition 7 in \cite{HaurayGomez}, there exists $Q\in\cL^1(\fH\otimes\fH)$ such that
$$
\Tr_{\fH\otimes\fH}(|Q_n-Q|)\to 0\quad\hbox{Êas }n\to\infty\,,
$$
for some subsequence of $Q_n$. Without loss of generality, we shall henceforth assume that the limit above holds for the whole sequence $Q_n$.

Since $Q_n\in\cC(K,K')$, one has $Q_n=Q_n^*\ge 0$, so that $Q=Q^*\ge 0$, and
$$
\Tr_{\fH\otimes\fH}(Q_n(A\otimes\rI_\fH+\rI_\fH\otimes B))=\Tr_\fH(KA)+\Tr_\fH(K'B)\,.
$$
Passing to the limit in the left hand side of the equality above as $n\to\infty$, one finds that
$$
\Tr_{\fH\otimes\fH}(Q(A\otimes\rI_\fH+\rI_\fH\otimes B))=\Tr_\fH(KA)+\Tr_\fH(K'B)
$$
for all bounded operators $A,B\in\cL(\fH)$, so that $Q\in\cC(K,K')$. 

Notice that the operator $\rI_{\fH\otimes\fH}+\eps C_\l(x,x',\grad_x,\grad_{x'})$ is unbounded self-adjoint, nonnegative and invertible on $\fH\otimes\fH$ for all $\eps>0$. Set
$$
C^\eps_\l(x,x',\grad_x,\grad_{x'}):=(\rI_{\fH\otimes\fH}+\eps C_\l(x,x',\grad_x,\grad_{x'}))^{-1}C_\l(x,x',\grad_x,\grad_{x'})\,.
$$
Obviously 
$$
0\le C^\eps_\l(x,x',\grad_x,\grad_{x'})=C^\eps_\l(x,x',\grad_x,\grad_{x'})^*\le\tfrac1\eps\rI_{\fH\otimes\fH}\,,
$$
so that
$$
\ba
\Tr_{\fH\otimes\fH}(Q_n^{1/2}C^\eps_\l(x,x',\grad_x,\grad_{x'})Q_n^{1/2})=\Tr_{\fH\otimes\fH}(Q_nC^\eps_\l(x,x',\grad_x,\grad_{x'}))
\\
\to\Tr_{\fH\otimes\fH}(QC^\eps_\l(x,x',\grad_x,\grad_{x'}))=\Tr_{\fH\otimes\fH}(Q^{1/2}C^\eps_\l(x,x',\grad_x,\grad_{x'})Q^{1/2})
\ea
$$
as $n\to+\infty$. On the other hand
$$
C^\eps_\l(x,x',\grad_x,\grad_{x'})\le C_\l(x,x',\grad_x,\grad_{x'})
$$
so that, for each $\eps>0$ and each $n\ge 1$, one has
$$
\ba
\Tr_{\fH\otimes\fH}(Q_n^{1/2}C^\eps_\l(x,x',\grad_x,\grad_{x'})Q_n^{1/2})\le\Tr_{\fH\otimes\fH}(Q_n^{1/2}C_\l(x,x',\grad_x,\grad_{x'})Q_n^{1/2})
\\
\to\MK_\l(K,K')^2
\ea
$$
as $n\to+\infty$. Hence
$$
\Tr_{\fH\otimes\fH}(Q^{1/2}C^\eps_\l(x,x',\grad_x,\grad_{x'})Q^{1/2})\le\MK_\l(K,K')^2
$$
for each $\eps>0$. In the limit as $\eps\to 0$, one has
$$
\Tr_{\fH\otimes\fH}(Q^{1/2}C^\eps_\l(x,x',\grad_x,\grad_{x'})Q^{1/2})\to\Tr_{\fH\otimes\fH}(Q^{1/2}C_\l(x,x',\grad_x,\grad_{x'})Q^{1/2})
$$
by monotone convergence, so that
$$
\Tr_{\fH\otimes\fH}(Q^{1/2}C_\l(x,x',\grad_x,\grad_{x'})Q^{1/2})\le\MK_\l(K,K')^2\,.
$$
Since $Q\in\cC(K,K')$, the inequality above is an equality, and $Q$ is a minimizer.
\end{proof}

\smallskip
Our first main result is the following theorem, which compares the pseudo-distance $\MK_\l$ for pairs of generalized T\"oplitz operators with the quadratic Monge-Kantorovich-Vasershtein distance between their symbols. 

\begin{Thm}\lb{T-UBound}
Let $R,R'\in\cD^2(\fH)$.

\smallskip
\noindent
(i) For all $\l>0$, one has
$$
\MK_\l(R^\l,(R')^\l)^2=\l\MK_1(R,R')^2\,.
$$
(ii) For all $\cq,\cq',\gp,\gp'\in\bR^d$ and each $\l>0$, one has
$$
\ba
\MK_\l(R^\l_{\cq,\gp/\l},(R')^\l_{\cq',\gp'/\l})^2=&|\cq-\cq'|^2+|\gp-\gp'|^2+\l\MK_1(R,R')^2
\\
&+2\sqrt\l\Tr_{L^2(\bR^d,dz)}((R-R')z)\cdot(\cq-\cq')
\\
&+2\sqrt{\l}Tr_{L^2(\bR^d,dz)}((R-R')(-i\l\grad_z))\cdot(\gp-\gp')\,.
\ea
$$
(iii) Let $\mu,\mu'$ be Borel probability measures on $\bR^d\times\bR^d$ satisfying the condition
$$
\int_{\bR^d\times\bR^d}(|p|^2+|q|^2)\mu(\dd p\dd q)+\int_{\bR^d\times\bR^d}(|p|^2+|q|^2)\mu'(\dd p\dd q)<\infty\,.
$$
Then
$$
\Op^R_\l[(2\pi\l)^d\mu]\hbox{ and }\Op^{R'}_\l[(2\pi\l)^d\mu']\in\cD^2(L^2(\bR^d))\,,
$$
and
$$
\ba
\MK_\l(\Op^R_\l[(2\pi\l)^d\mu],\Op^{R'}_\l[(2\pi\l)^d\mu'])^2\le\MKd(\mu,\mu')^2+\l\MK_1(R,R')^2
\\
+2\sqrt\lambda\int_{(\bR^d\times\bR^d)^2}\Tr_{L^2(\bR^d,dz)}((R-R')(-i\l\grad_z))\cdot\gp(\mu-\mu')(\dd\cq \dd\gp)
\\
+2\sqrt\lambda\int_{(\bR^d\times\bR^d)^2}\Tr_{L^2(\bR^d,dz)}((R-R')z)\cdot\cq(\mu-\mu')(\dd\cq \dd\gp)&\,.
\ea
$$
\end{Thm}

\begin{proof}
For each $\l>0$ and each $Q\in\cC(R,R')$, one has
$$
\ba
\Tr_{\fH\otimes\fH}(S_\l QS^*_\l(A\otimes\rI))=&\Tr_{\fH\otimes\fH}(QS^*_\l(A\otimes\rI)S_\l)
\\
=&\Tr_{\fH\otimes\fH}(Q((S^*_\l AS_\l)\otimes\rI))
\\
=&\Tr_\fH(RS^*_\l AS_\l)=\Tr_\fH(R^\l A)
\ea
$$
for each bounded operator on $\fH$, and, by the same token
$$
\Tr_{\fH\otimes\fH}(S_\l QS^*_\l(\rI\otimes A))=Tr_\fH((R')^\l A)\,.
$$
Hence $Q^\l=S_\l QS_\l^*$ runs through $\cC(R^\l,(R')^\l))$ as $Q$ runs through $\cC(R,R')$. 

Besides, straightforward computations show that
$$
S^*_\l C_\l(x,x',\grad_x,\grad_{x'})S_\l=\l C_1(x,x',\grad_x,\grad_{x'})
$$
so that
$$
\ba
\Tr_{\fH\otimes\fH}((Q^\l)^{1/2}C_\l(x,x',\grad_x,\grad_{x'})(Q^\l)^{1/2})
\\
=\l\Tr_{\fH\otimes\fH}(S_\l Q^{1/2}C_1(x,x',\grad_x,\grad_{x'})Q^{1/2}S_\l^*)
\\
=\l\Tr_{\fH\otimes\fH}(Q^{1/2}C_1(x,x',\grad_x,\grad_{x'})Q^{1/2})&\,,
\ea
$$
since $S_\l^*=S_\l^{-1}$ on $\fH\otimes\fH$. Thus
$$
\ba
\MK_\l(R^\l,(R')^\l)^2=\inf_{Q\in\cC(R,R')}\Tr_{\fH\otimes\fH}((Q^\l)^{1/2}C_\l(x,x',\grad_x,\grad_{x'})(Q^\l)^{1/2})
\\
=\l\inf_{Q\in\cC(R,R')}\Tr_{\fH\otimes\fH}(Q^{1/2}C_1(x,x',\grad_x,\grad_{x'})Q^{1/2})
\\
=\l\MK_1(R,R')^2&\,.
\ea
$$
This proves statement (i).

For each $\cq,\cq',\gp,\gp'\in\bR^d$ and each $Q\in\cC(R,R')$, set
$$
Q^\l_{\cq,\cq',\gp/\l,\gp'/\l}:=T_{(\cq,\cq'),(\gp/\l,\gp'/\l)}S_\l QS^*_\l T^*_{(\cq,\cq'),(\gp/\l,\gp'/\l)}\,.
$$
Obviously, for each bounded operator $A$ on $\fH$, one has
$$
\ba
\Tr_{\fH\otimes\fH}(Q^\l_{\cq,\cq',\gp/\l,\gp'/\l}(A\otimes\rI))
\\
=\Tr_{\fH\otimes\fH}(QS^*_\l T^*_{(\cq,\cq'),(\gp/\l,\gp'/\l)}(A\otimes\rI)T_{(\cq,\cq'),(\gp/\l,\gp'/\l)}S_\l)
\\
=\Tr_{\fH\otimes\fH}(Q((S^*_\l T^*_{\cq,\gp/\l}AT_{\cq,\gp/\l}S_\l)\otimes\rI))
\\
=\Tr_\fH(R(S^*_\l T^*_{\cq,\gp/\l}AT_{\cq,\gp/\l}S_\l))
\\
=\Tr_\fH(R^\l_{\cq,\gp/\l}A)&\,,
\ea
$$
and by the same token
$$
\Tr_{\fH\otimes\fH}(Q^\l_{\cq,\cq',\gp/\l,\gp'/\l}(\rI\otimes A))=\Tr_\fH((R')^\l_{\cq',\gp'/\l}A)\,.
$$
Hence $Q^\l_{\cq,\cq',\gp/\l,\gp'/\l}\in\cC(R^\l_{\cq,\gp/\l},(R')^\l_{\cq',\gp'/\l})$. Moreover, the argument above shows that $Q^\l_{\cq,\cq',\gp/\l,\gp'/\l}$ runs through $\cC(R^\l_{\cq,\gp/\l},(R')^\l_{\cq',\gp'/\l})$ as $Q$ runs through $\cC(R,R')$. 

By a straightforward computation, 
$$
\ba
T^*_{(\cq,\cq'),(\gp/\l,\gp'/\l)}C_\l(x,x',\grad_x,\grad_{x'})T_{(\cq,\cq'),(\gp/\l,\gp'/\l)}=&|\cq-\cq'|^2+|\gp-\gp'|^2
\\
&+2(\cq-\cq')\cdot(x-x')
\\
&+2(\gp-\gp')\cdot(i\l\grad_y-i\l\grad_x)
\\
&+C_\l(x,x',\grad_x,\grad_{x'})\,.
\ea
$$
Hence
$$
\ba
\Tr_{\fH\otimes\fH}((Q^\l_{\cq,\cq',\gp/\l,\gp'/\l})^{1/2}C_\l(x,x',\grad_x,\grad_{x'})(Q^\l_{\cq,\cq',\gp/\l,\gp'/\l})^{1/2})
\\
=|\cq-\cq'|^2+|\gp-\gp'|^2+\Tr_{\fH\otimes\fH}((Q^\l)^{1/2}C_\l(x,x',\grad_x,\grad_{x'})(Q^\l)^{1/2})
\\
+2(\gp-\gp')\cdot\Tr_{\fH\otimes\fH}(-i\l(\grad_{x}-\grad_{x'})Q)
\\
+2(\cq-\cq')\cdot\Tr_{\fH\otimes\fH}((x-x')Q)&\,.
\ea
$$
Observe that
$$
\ba
\Tr_{\fH\otimes\fH}((x-x')Q^\l)=\Tr_\fH(xR^\l)-\Tr_\fH(x'(R')^\l)
\\
=\sqrt\lambda\Tr_{L^2(\bR^d,dz)}(z(R-R'))&\,,
\ea
$$
while
$$
\ba
\Tr_{\fH\otimes\fH}(-i\l(\grad_{x}-i\l\grad_{x'})Q^\l)=\Tr_\fH(i\l\grad_xR^\l)-\Tr_\fH(-i\l\grad_{x'}(R')^\l)&
\\
=\sqrt\lambda\Tr_{L^2(\bR^d,dz)}(-i\l\grad_z(R-R'))&\,,
\ea
$$
since $S_\lambda x S_\lambda^*=\sqrt\lambda x$ and  $S_\lambda (-i\l\grad_x) S_\lambda^*=\lambda^{-1/2} (-i\l\grad_x)$. 
Therefore
\be\lb{2points}
\ba
\Tr_{\fH\otimes\fH}((Q^\l_{\cq,\cq',\gp/\l,\gp'/\l})^{1/2}C_\l(x,x',\grad_x,\grad_{x'})(Q^\l_{\cq,\cq',\gp/\l,\gp'/\l})^{1/2})
\\
=|\cq-\cq'|^2+|\gp-\gp'|^2+\Tr_{\fH\otimes\fH}((Q^\l)^{1/2}C_\l(x,x',\grad_x,\grad_{x'})(Q^\l)^{1/2})
\\
+2\sqrt\l(\gp-\gp')\cdot\Tr_{L^2(\bR^d,dz)}(-i\l\grad_z(R-R'))
\\
+2\sqrt\l(\cq-\cq')\cdot\Tr_{L^2(\bR^d,dz)}(z(R-R'))&\,.
\ea
\ee
We have seen that $Q^\l_{\cq,\cq',\gp/\l,\gp'/\l}$ runs through $\cC(R^\l_{\cq,\gp/\l},(R')^\l_{\cq',\gp'/\l})$ while $Q^\l$ runs through $\cC(R^\l,(R')^\l)$ as $Q$ runs through $\cC(R,R')$; thus
$$
\ba
\MK_\l(R^\l_{\cq,\gp/\l},(R')^\l_{\cq',\gp'/\l})^2
\\
=
\inf_{Q\in\cC(R,R')}\Tr_{\fH\otimes\fH}((Q^\l_{\cq,\cq',\gp/\l,\gp'/\l})^{1/2}C_\l(x,x',\grad_x,\grad_{x'})(Q^\l_{\cq,\cq',\gp/\l,\gp'/\l})^{1/2})
\\
=|\cq-\cq'|^2+|\gp-\gp'|^2+\inf_{Q\in\cC(R,R')}\Tr_{\fH\otimes\fH}((Q^\l)^{1/2}C_\l(x,x',\grad_x,\grad_{x'})(Q^\l)^{1/2})
\\
+2\sqrt\l(\gp-\gp')\cdot\Tr_{L^2(\bR^d,dz)}(-i\l\grad_x(R-R'))
\\
+2\sqrt\l(\cq-\cq')\cdot\Tr_{L^2(\bR^d,dz)}(z(R-R'))
\\
=|\cq-\cq'|^2+|\gp-\gp'|^2+\MK_\l(R^\l,(R')^\l)^2
\\
+2\sqrt\l(\gp-\gp')\cdot\Tr_{L^2(\bR^d,dz)}(-i\l\grad_x(R-R'))
\\
+2\sqrt\l(\cq-\cq')\cdot\Tr_{L^2(\bR^d,dz)}(z(R-R'))&\,,
\ea
$$
With the formula in statement (i), this implies statement (ii).

Let $Q\in\cC(R,R')$, and let $\rho$ be an optimal coupling of $\mu$ and $\mu'$, i.e. $\rho$ is a Borel probability measure on $(\bR^d\times\bR^d)^2$ satisfying
$$
\ba
\int_{(\bR^d\times\bR^d)^2}(f(q,p)+g(q',p'))\rho(\dd p\dd q\dd p'\dd q')=&\int_{\bR^d\times\bR^d}f(q,p)\mu(\dd p\dd q)
\\
&+\int_{\bR^d\times\bR^d}g(q',p')\mu'(\dd p'\dd q')
\ea
$$
for all $f,g\in C_b(\bR^d\times\bR^d)$, and
$$
\MKd(\mu,\mu')^2=\int_{(\bR^d\times\bR^d)^2}(|q-q'|^2+\l^2|p-p'|^2)\rho(\dd p\dd q\dd p'\dd q')\,.
$$
Set
$$
\cQ^\l:=\int_{(\bR^d\times\bR^d)^2}Q^\l_{\cq,\cq',\gp/\l,\gp'/\l}\rho(\dd \cq\dd \cq'\dd\gp\dd\gp')\,.
$$
Then, for each bounded operator $A$ on $\fH$, one has
$$
\ba
\Tr_{\fH\otimes\fH}(\cQ^\l(A\otimes\rI))
\\
=\int_{(\bR^d\times\bR^d)^2}\Tr_{\fH\otimes\fH}(Q^\l_{\cq,\cq',\gp/\l,\gp'/\l}(A\otimes\rI))\rho(\dd \cq\dd \cq'\dd\gp\dd\gp')
\\
=\int_{(\bR^d\times\bR^d)^2}\Tr_{\fH}(R^\l_{\cq,\gp/\l}A)\rho(\dd \cq\dd \cq'\dd\gp\dd\gp')
\\
=\int_{\bR^d\times\bR^d}\Tr_{\fH}(R^\l_{\cq,\gp/\l}A)\mu(\dd \cq\dd\gp)
\\
=\Tr_\fH\left(A\int_{\bR^d\times\bR^d}R^\l_{\cq,\gp/\l}\mu(\dd \cq\dd\gp)\right)
\\
=\Tr_\fH(\Op^R_\l[(2\pi\l)^d\mu]A)&\,.
\ea
$$
By the same token
$$
\Tr_{\fH\otimes\fH}(\cQ^\l(\rI\otimes A))=\Tr_\fH(\Op^R_\l[(2\pi\l)^d\mu']A)\,,
$$
so that
$$
\cQ^\l\in\cC(\Op^R_\l[(2\pi\l)^d\mu],\Op^R_\l[(2\pi\l)^d\mu'])\,.
$$
Integrating both sides of formula (\ref{2points}) with respect to the measure $\rho$, one finds by \eqref{2points} that
\be\lb{2points2}
\ba
\Tr_{\fH\otimes\fH}((\cQ^\l)^{1/2}C_\l(x,x',\grad_x,\grad_{x'})(\cQ^\l)^{1/2})
\\
=\!\!\!\int_{\bR^{4d}}\!\!\Tr_{\fH\otimes\fH}\left(\sqrt{Q^\l_{\cq,\cq',\gp/\l,\gp'/\l}}C_\l(x,x',\grad_x,\grad_{x'})\sqrt{Q^\l_{\cq,\cq',\gp/\l,\gp'/\l}}\right)\rho(\dd \cq\dd\gp\dd \cq'\dd\gp')
\\
=\int_{\bR^{4d}}(|\cq-\cq'|^2+|\gp-\gp'|^2)\rho(\dd \cq\dd\gp\dd \cq'\dd\gp')
\\
+2\sqrt\l\int_{\bR^{4d}}(\cq-\cq')\cdot\Tr_{L^2(\bR^d,dz)}(z(R-R'))\rho(\dd \cq\dd\gp\dd \cq'\dd\gp')
\\
+2\sqrt\l\int_{\bR^{4d}}(\gp-\gp')\cdot\Tr_{L^2(\bR^d,dz)}(-i\l\grad_z(R-R'))\rho(\dd \cq\dd\gp\dd \cq'\dd\gp')
\\
+\int_{\bR^{4d}}\Tr_{\fH\otimes\fH}((Q^\l)^{1/2}C_\l(x,x',\grad_x,\grad_{x'})(Q^\l)^{1/2})\rho(\dd \cq\dd\gp\dd \cq'\dd\gp')
\\
=\MKd(\mu,\mu')^2+\Tr_{\fH\otimes\fH}((Q^\l)^{1/2}C_\l(x,x',\grad_x,\grad_{x'})(Q^\l)^{1/2})
\\
+2\sqrt\l\int_{\bR^{4d}}(\gp-\gp')\cdot\Tr_{L^2(\bR^d,dz)}(-i\l\grad_z(R-R'))\rho(\dd \cq\dd\gp\dd \cq'\dd\gp')
\\
+2\sqrt\l\int_{\bR^{4d}}(\cq-\cq')\cdot\Tr_{L^2(\bR^d,dz)}(z(R-R'))\rho(\dd \cq\dd\gp\dd \cq'\dd\gp')&\,.
\ea
\ee
Minimizing both sides of this equality as $Q$ runs through $\cC(R,R')$, we see that
$$
\ba
\MK_\l((\Op^R_\l[(2\pi\l)^d\mu],\Op^R_\l[(2\pi\l)^d\mu']))^2\!\le\!\MKd(\mu,\mu')^2\!+\!\MK_\l(R^\l,(R')^\l)^2
\\
+2\sqrt\l\int_{\bR^{4d}}(\gp-\gp')\cdot\Tr_{L^2(\bR^d,dz)}(-i\l\grad_z(R-R'))\rho(\dd \cq\dd\gp\dd \cq'\dd\gp')
\\
+2\sqrt\l\int_{\bR^{4d}}(\cq-\cq')\cdot\Tr_{L^2(\bR^d,dz)}(z(R-R'))\rho(\dd \cq\dd\gp\dd \cq'\dd\gp')&\,.
\ea
$$
Finally, we use statement (i) to express the last term on the right hand side as
$$
\MK_\l(R^\l,(R')^\l)^2=\l\MK_1(R,R')^2\,,
$$
and this concludes the proof.
\end{proof}

\smallskip
Several remarks are in order after Theorem \ref{T-UBound}. First we recall formula (14) from \cite{FGMouPaul}: for each $R,R'\in\cD^2(L^2(\bR^d))$, one has
\be\lb{MK>2d}
MK_1(R,R')^2\ge 2d\quad\hbox{ for all }R,R'\in\cD^2(L^2(\bR^d))\,.
\ee

\begin{Cor}\lb{C-Gauss}
Let $a$ be the Gaussian state (\ref{Gauss}). The corresponding density operator $|a\ra\la a|=|0,0,1,a\ra\la 0,0,1,a|$ minimizes the $\MK_1$ (pseudo-)distance to itself,  i.e.
$$
\MK_1(|a\ra\la a|\,,\,|a\ra\la a|)^2=2d\,.
$$
An optimal coupling of $|a\ra\la a|$ with itself is
$$
|a\ra\la a|\otimes|a\ra\la a|\,.
$$
More generally, for all $q,q',p,p'\in\bR^d$ and $\l>0$, one has
$$
\MK_\l(|q,p,\l,a\ra\la q,p,\l,a|\,,\,|q',p',\l,a\ra\la q',p',\l,a|)^2=|q-q'|^2+|p-p'|^2+2d\l\,.
$$
\end{Cor}

\begin{proof}
Applying Theorem 2.3 (1) in \cite{FGMouPaul} with $\eps=1$ and $\mu_1=\mu_2=\de_{(0,0)}$ shows that
$$
\MK_1(|a\ra\la a|\,,\,|a\ra\la a|)^2\le 2d\,.
$$
The reverse inequality follows from (\ref{MK>2d}).

The optimality of the coupling 
$$
|a\ra\la a|\otimes|a\ra\la a|\,.
$$
of $|a\ra\la a|$ with itself follows from formula (30) in \cite{FGMouPaul} with $\mu=\de_{(0,0)}\otimes\de_{(0,0)}$. 

The second equality in the corollary follows from the first, together with the identity in Theorem \ref{T-UBound} (ii).
\end{proof}

\smallskip
The first equality in Corollary \ref{C-Gauss} shows that the transport from the Gaussian density $|a\ra\la a|$ to itself minimizes the pseudo-distance $\MK_1$. In fact, there is a much wider class of densities enjoying the same property.

\begin{Cor}\lb{C-MinMK}
Let $R\in\cD^2(L^2(\bR^d))$ satisfy the minimality condition
$$
\MK_1(R,R)^2=2d\,.
$$
Then, for all each Borel probability measure $\mu$ on $\bR^d\times\bR^d$ with finite second order moment, i.e. satisfying
$$
\iint_{\bR^d\times\bR^d}(|q|^2+|p|^2)\mu(\dd q\dd p)<\infty\,,
$$
one has
$$
\MK_\l(\Op^R_\l[(2\pi\l)^d\mu],\Op^R_\l[(2\pi\l)^d\mu])^2=2d\l\,,
$$
\end{Cor}

\begin{proof}
That 
$$
\MK_\l(\Op^R_\l[(2\pi\l)^d\mu],\Op^R_\l[(2\pi\l)^d\mu])^2\ge 2d\l
$$
follows from formula (14) in \cite{FGMouPaul}, or from formula (\ref{MK>2d}) and Theorem \ref{T-UBound} (i). On the other hand, by Theorem \ref{T-UBound} (iii)
$$
\MK_\l(\Op^R_\l[(2\pi\l)^d\mu],\Op^R_\l[(2\pi\l)^d\mu])^2\le\MKd(\mu,\mu)^2+\l\MK_1(R,R)^2=2d\l\,.
$$
\end{proof}

\smallskip
Corollary \ref{C-Gauss} shows that any classical T\"oplitz operator $\Op^T_{\hbar}[(2\pi)^d\mu]$, where $\mu$ is a Borel probability measure on $\bR^d\times\bR^d$ with finite second order moment, minimizes the pseudo-distance 
$\MK_{\hbar}$ to itself i.e. $\MK_1(\Op^T_{\hbar}[(2\pi)^d\mu],\Op^T_{\hbar}[(2\pi)^d\mu])^2=2d\hbar$.
 
 In fact, one can easily characterize the density operators minimizing the $\MK_1$ (pseudo-)distance to themselves: they must be the marginals of any fundamental state of the operator $C_1(x,x',\grad_x,\grad_{x'})$. More precisely,
 one has the following characterization.
 
 \begin{Prop}\lb{P-MinMK1}
 Let $R\in\cD^2(L^2(\bR^d))$. Then 
 $$
 \MK_1(R,R)=2d
 $$
 if and only if there exist $\rho\equiv\rho(z,z')\in L^2(\bR^d\times\bR^d)$ such that the operator with integral kernel $\rho$ is self-adjoint nonnegative and trace-class on $L^2(\bR^d)$, and the integral kernel $r(x,x')$ of $R$ is given 
 by the expression
 \be\lb{Form-r}
 r(x,x')=\int_{\bR^d}e^{-(|x-z|^2+|x'-z|^2)/4}\rho\left(\frac{x+z}2,\frac{x'+z}2\right)\dd z\,.
 \ee
 \end{Prop}
 
 \smallskip
 An obvious consequence of the proposition is the following ``separation'' property.
 
 \begin{Cor}\lb{C-Sep}
 In particular, for each $R,R'\in\cD^2(L^2(\bR^d))$, one has
 $$
 R\neq R'\Longrightarrow \MK_1(R,R')>2d\,.
 $$
 \end{Cor}
 
 \smallskip
 Notice however that the converse of the implication in Corollary \ref{C-Sep} is not true, as can be seen from Proposition \ref{P-MinMK1}.
 
\begin{proof}[Proof of Proposition \ref{P-MinMK1}]
Let us assume that $\MK_1(R,R)=2d$. By Lemma \ref{L-Inf=Min}, there exists $Q\in\cC(R,R)$ such that
\be\label{premiere}
\Tr_{L^2(\bR^d)\otimes L^2(\bR^d)}(Q^{1/2}C_1(x,y,\grad_x,\grad_y)Q^{1/2})=2d\,.
\ee
Observing that
 $$
(x_j-y_j)^2-(\d_{x_j}-\d_{y_j})^2-2=\left((x_j-y_j)-(\d_{x_j}-\d_{y_j})\right)\left((x_j-y_j)+(\d_{x_j}-\d_{y_j})\right)\,,
$$
we conclude that
$$
A=\left((x_j-y_j)+(\d_{x_j}-\d_{y_j})\right)Q^{1/2}=0\,,
$$
since \eqref{premiere} can be put in the form
$$
\Tr_{L^2(\bR^d)\otimes L^2(\bR^d)}(A^*A)=0\,.
$$

\noindent Hence, the integral kernel $u\equiv u(x,y,x',y')$ of $Q^{1/2}$ is of the form
$$
u(x,y,x',y')=e^{-|x-y|^2/4}s\left(\frac{x+y}2,x',y'\right)\,,
$$
with $s\in L^2((\bR^d)^3)$. Since $Q$ is self-adjoint, so is $Q^{1/2}$. Therefore the integral kernel of $Q$ is of the form
\be\lb{Form-q}
q(x,y,x',y')=e^{-(|x-y|^2+|x'-y'|^2)/4}\rho\left(\frac{x+y}2,\frac{x'+y'}2\right)\,,
\ee
with
$$
\rho(z,z'):=\iint_{\bR^d\times\bR^d}s(z,x'',y'')s(z',x'',y'')\dd x''\dd y''\,.
$$
By construction, $\rho$ is the integral kernel of a nonnegative, self-adjoint, trace-class operator on $L^2(\bR^d)$. (That the operator with integral kernel $\rho$ is trace-class on $L^2(\bR^d)$ follows form the fact that $s\in L^2((\bR^d)^3)$). 
Since $R$ is the first (or the second) marginal of $Q$, its integral kernel must be given by the formula
$$
r(x,x')=\int_{\bR^d}q(x,z,x',z)\dd z\,.
$$
With the expression (\ref{Form-q}) for $q$, this is equivalent to the formula (\ref{Form-r}) for $r$ in the statement of the proposition.

Conversely, let $R\in\cD^2(L^2(\bR^d))$ be defined in terms of an integral kernel $r$ of the form as in the proposition. Defining $q$ by formula (\ref{Form-q}) in terms of the function $\rho$ provided by the proposition, we see that the operator
$Q$ with integral kernel $q$ is self-adjoint and nonnegative on $L^2((\bR^d)^2)$, because the operator with integral kernel $\rho$ is self-adjoint nonnegative on $L^2(\bR^d)$. That $Q\in\cC(R,R)$ follows from the symmetry of the kernel $\rho$
and formula (\ref{Form-r}). With $Q$ defined in this way, one has
$$
\MK_1(R,R)^2\le\Tr_{L^2(\bR^d)\otimes L^2(\bR^d)}(Q^{1/2}C_1(x,y,\grad_x,\grad_y)Q^{1/2})=2d\,.
$$
With the reverse inequality (\ref{MK>2d}), we conclude that if $r$ is given by formula (\ref{Form-r}), then $\MK_1(R,R)^2=2d$.
\end{proof}

\begin{proof}[Proof of Corollary \ref{C-Sep}]
If $\MK_1(R,R')=2d$, there exists a coupling $Q\in\cC(R,R')$ such that
$$
\Tr_{L^2(\bR^d)\otimes L^2(\bR^d)}(Q^{1/2}C_1(x,y,\grad_x,\grad_y)Q^{1/2})=2d
$$
by Lemma \ref{L-Inf=Min}. Arguing as in the proof of Proposition \ref{P-MinMK1}, we conclude that $q$ must be of the form (\ref{Form-q}). This implies that 
$$
q(x,y,x',y')=q(y,x,y',x')\quad\hbox{ for a.e. }x,y,x',y'\in\bR^d\,.
$$
Hence the integral kernels $r$ and $r'$ of $R$ and $R'$ respectively satisfy
$$
r(x,x')=\int_{\bR^d}q(x,z,x',z)\dd z=\int_{\bR^d}q(z,x,z,x')\dd z=r'(x,x')
$$
for a.e. $x,x'\in\bR^d$, so that $R=R'$.
\end{proof}

\medskip
Theorem \ref{T-UBound} provides a control of $\MK_\l(K,K)^2$ in the case where $K$ and $K'$ are generalized T\"oplitz operators, in terms of the symbols of these operators. 

However, Theorem \ref{T-UBound} does not apply to general density operators. The following observation provides an alternative control of $\MK_\l(K,K')$ in terms of the Wigner functions of $K$ and $K'$ respectively, and therefore 
does apply to a larger class of density operators. 

\begin{Prop}\label{propappen}
Consider two families of density matrices $\rho_\l,\rho'_\l\in\cD^2(L^2(\bR^d))$ (not necessarily generalized T\"oplitz operators) indexed by $\l>0$. Then, for all $\l>0$, one has
$$
\MK_\l(\rho_\l,\rho'_\l)^2\le\int_{\bR^{4d}}(|q-q'|^2+|p-p'|^2)W_\l[\rho_\l](q,p)W_\l[\rho'_\l](q',p')\dd q\dd p\dd q'\dd p'\,,
$$
where $W_\l[\rho_\l]$ and $W_\l[\rho'_\l]$ are the Wigner functions of $\rho_\l$ and $\rho'_\l$ respectively, as defined in \eqref{DefW}
\end{Prop}

\begin{proof}
Since $\rho_\l\otimes\rho_\l$ is a coupling of $\rho_\l$ and $\rho'_\l$, one has
$$
\MK_\l(\rho_\l,\rho'_\l)^2\leq\Tr_{L^2(\bR^d)\otimes L^2(\bR^d)}\left((\rho_\l\otimes\rho'_\l)^{1/2}C_\l(x,x',\grad_x,\grad_{x'})(\rho_\l\otimes\rho'_\l)^{1/2}\right)\,.
$$
Next, one has
$$
W_{\l}[\rho_\l\otimes\rho'_\l]=W_{\hbar}[\rho_\l]\otimes W_{\hbar}[\rho'_\l]\,.
$$
Denoting by $r_\l\equiv r_\l(X,Y)$ and $r'_\l\equiv r'_\l(X',Y')$ the integral kernels of $\rho_\l$ and $\rho'_\l$ respectively, one has
$$
\int_{\bR^{2d}}W_\l[\rho_\l](q,p)W_\l[\rho'](q',p')\dd p\dd p'=r_\l(q,q)r'_\l(q',q')\,,
$$
and
$$
\int_{\bR^{2d}}W_\l[\rho_\l](q,p)W_\l[\rho'](q',p')\dd q\dd q'=\frac1{(2\pi\l)^{2d}}\hat r_\l\left(\frac{p}\l,\frac{p}\l\right)\hat r'_\l\left(\frac{p'}\l,\frac{p'}\l\right)\,,
$$
where $\hat r_\l$ and $\hat r'_\l$ are the twisted Fourier transforms of $r_\l$ and $r'_\l$ respectively, i.e.
$$
\ba
\hat r(\xi,\eta):=\iint_{\bR^d\times\bR^d}r_\l(x,y)e^{-i(\xi\cdot x-\eta\cdot y)}\dd x\dd y\,,
\\
\hat r'(\xi,\eta):=\iint_{\bR^d\times\bR^d}r'_\l(x,y)e^{-i(\xi\cdot x-\eta\cdot y)}\dd x\dd y\,.
\ea
$$
Hence
$$
\ba
\Tr_{L^2(\bR^d)\otimes L^2(\bR^d)}\left((\rho_\l\otimes\rho'_\l)^{1/2}|x-x'|^2(\rho_\l\otimes\rho'_\l)^{1/2}\right)
\\
=\int_{\bR^{4d}}|q-q'|^2W_\l[\rho_\l](q,p)W_\l[\rho'_\l](q',p')\dd q\dd p\dd q'\dd p'&\,,
\ea
$$
while
$$
\ba
\Tr_{L^2(\bR^d)\otimes L^2(\bR^d)}\left((\rho_\l\otimes\rho'_\l)^{1/2}(\grad_x-\grad_{x'})\cdot(\grad_x-\grad_{x'})(\rho_\l\otimes\rho'_\l)^{1/2}\right)
\\
=-\frac1{\l^2}\int_{\bR^{4d}}|p-p'|^2W_\l[\rho_\l](q,p)W_\l[\rho'_\l](q',p')\dd q\dd p\dd q'\dd p'&\,.
\ea
$$
Hence
$$
\ba
\Tr_{L^2(\bR^d)\otimes L^2(\bR^d)}\left((\rho_\l\otimes\rho'_\l)^{1/2}C_\l(x,x',\grad_x,\grad_{x'})(\rho_\l\otimes\rho'_\l)^{1/2}\right)
\\
=
\int_{\bR^{4d}}(|q-q'|^2+|p-p'|^2)W_\l[\rho_\l](q,p)W_\l[\rho'_\l](q',p')\dd q\dd p\dd q'\dd p'&\,,
\ea
$$
and this concludes the proof.
\end{proof}

\smallskip
Thus, if the families of density operators $\rho_\l$ and $\rho'_\l$ satisfy 
$$
W_\l[\rho_\l]\to\de_{q_0,p_0}\quad\hbox{ and }\quad W_\l[\rho'_\l]\to\de_{q_0,p_0}
$$
in the sense of distributions as $\l\to 0^+$, together with appropriate tightness conditions, then
$$
\int_{\bR^{4d}}(|q-q'|^2+|p-p'|^2)W_\l[\rho_\l](q,p)W_\l[\rho'_\l](q',p')\dd q\dd p\dd q'\dd p'\to 0
$$
as $\l\to 0$ with some convergence rate, and the inequality in the proposition above implies that
$$
\MK_\l(\rho_\l,\rho'_\l)\to 0\quad\hbox{ as }\l\to 0\,,
$$
with the same convergence rate.


\section{A Lower Bound for $\MK_{\hbar}$}


The next theorem generalizes statement (2) in Theorem 2.3 of \cite{FGMouPaul} to the positive quantization in Definition \ref{D-Opmu}.

\begin{Thm}\lb{T-LBound}
Let $R,R',K,K'\in\cD^2(\fH)$. For each $\l>0$, one has
$$
\ba
\MK_\l(K,K')^2\ge&\MKd(\tilde W^R_\l[K],\tilde W^{R'}_{\l}[K'])^2-\l\MK_1(R,R')^2
\\
&+2\sqrt\l\Tr_{L^2(\bR^d,dz)}((R-R')z)\cdot\Tr_{L^2(\bR^d,dy)}(y(K-K'))
\\
&-2\l^{3/2}\Tr_{L^2(\bR^d,dz)}((R-R')\grad_z)\cdot\Tr_{L^2(\bR^d,dy)}(\grad_y(K-K'))\,.
\ea
$$
\end{Thm}

\medskip
We begin with two elementary computations. The first lemma below is the analogue of formula (48) in \cite{FGMouPaul}.

\begin{Lem}\lb{L-OpaCost} Let $R,R'\in\cD_2(\fH)$, and let $Q\in\cC(R,R')$. For each $\l>0$
$$
\ba
\frac1{(2\pi\l)^{2d}}\int_{(\bR^d\times\bR^d)^2}(|q-q'|^2+|p-p'|^2)Q^\l_{q,q',p/\l,p'/\l}\dd p\dd q\dd p'\dd q'
\\
=|x-x'|^2-\l^2|\grad_x-\grad_{x'}|^2+\l\Tr_{\fH\otimes\fH}(Q^{1/2}C_1Q^{1/2})\rI_{\fH\otimes\fH}
\\
\pm2\sqrt\l\Tr_{L^2(\bR^d,dz)}((R-R')z)\cdot(x-x')
\\
\pm2\sqrt\l\Tr_{L^2(\bR^d,dz)}((R-R')(-i\grad_z))\cdot(-i\l(\grad_x-\grad_{x'}))&\,.
\ea
$$
\end{Lem}

\begin{proof}[Proof of Lemma \ref{L-OpaCost}]
Denote by $a\equiv a(X,X',Y,Y')\in\bC$ the integral kernel of the operator $Q$. For each $\l>0$ and each $q,q',p,p'\in\bR^d$, the the integral kernel of the operator $Q^\l_{q,q',p/\l,p'/\l}$ is
$$
\l^{-d}a\left(\frac{x-q}{\sqrt\l},\frac{x'-q}{\sqrt\l},\frac{y-q}{\sqrt\l},\frac{y'-q}{\sqrt\l}\right)e^{i(p\cdot(x-y)+p'\cdot(x'-y'))/\l}\,.
$$
Thus the integral kernel of the operator
$$
\frac1{(2\pi\l)^{2d}}\int_{(\bR^d\times\bR^d)^2}|q-q'|^2Q^\l_{q,q',p/\l,p'/\l}\dd p\dd q\dd p'\dd q'
$$
is
$$
\ba
\int_{(\bR^d\times\bR^d)^2}|q-q'|^2\l^{-d}a\left(\tfrac{x-q}{\sqrt\l},\tfrac{x'-q'}{\sqrt\l},\tfrac{y-q}{\sqrt\l},\tfrac{y'-q'}{\sqrt\l}\right)e^{i(p\cdot(x-y)+p'\cdot(x'-y'))/\l}\frac{\dd p\dd q\dd p'\dd q'}{(2\pi\l)^{2d}}
\\
=\left(\int_{(\bR^d\times\bR^d)^2}|q-q'|^2\l^{-d}a\left(\tfrac{x-q}{\sqrt\l},\tfrac{x'-q'}{\sqrt\l},\tfrac{x-q}{\sqrt\l},\tfrac{x'-q'}{\sqrt\l}\right)\dd q\dd q'\right)\de(x-y)\de(x'-y')
\\
=\left(\int_{\bR^d\times\bR^d}|(x\!-\!x')\!-\!\sqrt\l(X\!-\!X')|^2a(X,X',X,X')\dd X\dd X'\right)\de(x\!-\!y)\de(x'\!-\!y')
\\
=\left(|x-x'|^2-2\sqrt\lambda\Tr_{L^2(\bR^d,dz)}((R-R')z)\cdot(x-x')\right.
\\
\left.
+
\l\int_{\bR^d\times\bR^d}|X-X'|^2a(X,X',X,X')dXdX'\right)\de(x-y)\de(x'-y')&\,.
\ea
$$
To obtain the second term in the last right hand side, we have used the identity
$$
\ba
\int_{\bR^d\times\bR^d}(X-X')a(X,X',X,X')\dd X\dd X'=\Tr_{\fH\otimes\fH}((X-X')Q)
\\
=\Tr_\fH(XR)-\Tr_\fH(X'R')=\Tr_{L^2(\bR^d,dz)}((R-R')z)&\,.
\ea
$$
In other words
$$
\ba
\frac1{(2\pi\l)^{2d}}\int_{(\bR^d\times\bR^d)^2}|q-q'|^2Q^\l_{q,q',p/\l,p'/\l}dpdqdp'dq'
\\
=
|x-x'|^2+\l\Tr_{\fH\otimes\fH}(|X-X'|^2Q)\rI_{\fH\otimes\fH}
\\
-2\sqrt\l\Tr_{L^2(\bR^d,dz)}((R-R')z)\cdot(x-x')&\,.
\ea
$$

Next, the integral kernel of the operator
$$
-\frac1{(2\pi\l)^{2d}}\int_{(\bR^d\times\bR^d)^2}|p-p'|^2Q^\l_{q,q',p/\l,p'/\l}\dd p\dd q\dd p'\dd q'
$$
is
$$
\ba
-\int_{(\bR^d\times\bR^d)^2}|p-p'|^2\l^{-d}a\left(\tfrac{x-q}{\sqrt\l},\tfrac{x'-q'}{\sqrt\l},\tfrac{y-q}{\sqrt\l},\tfrac{y'-q'}{\sqrt\l}\right)e^{i(p\cdot(x-y)+p'\cdot(x'-y'))/\l}\frac{\dd p\dd q\dd p'\dd q'}{(2\pi\l)^{2d}}
\\
=
\left(\int_{\bR^d\times\bR^d}a\left(\tfrac{x-q}{\sqrt\l},\tfrac{x'-q'}{\sqrt\l},\tfrac{y-q}{\sqrt\l},\tfrac{y'-q'}{\sqrt\l}\right)\tfrac{\dd q\dd q'}{\l^d}\right)\l^2(\grad_x\!-\!\grad_{x'})\!\cdot\!(\grad_y\!-\!\grad_{y'})\de(x\!-\!y)\de(x'\!-\!y')
\ea
$$
in the sense of (tempered) distributions on $(\bR^d\times\bR^d)^2$. The integral on the right hand side can be put in the form
$$
\ba
-\int_{\bR^{4d}}|p-p'|^2\l^{-d}a\left(\tfrac{x-q}{\sqrt\l},\tfrac{x'-q'}{\sqrt\l},\tfrac{y-q}{\sqrt\l},\tfrac{y'-q'}{\sqrt\l}\right)e^{i(p\cdot(x-y)+p'\cdot(x'-y'))/\l}\frac{\dd p\dd q\dd p'\dd q'}{(2\pi\l)^{2d}}
\\
=
\l^2(\grad_x\!-\!\grad_{x'})\cdot(\grad_y\!-\!\grad_{y'})\left(\left(\int_{\bR^d\times\bR^d}a\left(X,X',X,X'\right)\dd X\dd X'\right)\de(x\!-\!y)\de(x'\!-\!y')\right)
\\
-\left(\int_{\bR^{2d}}\l^2(\grad_x\!-\!\grad_{x'})\cdot(\grad_y\!-\!\grad_{y'})a\left(\tfrac{x-q}{\sqrt\l},\tfrac{x'-q'}{\sqrt\l},\tfrac{y-q}{\sqrt\l},\tfrac{y'-q'}{\sqrt\l}\right)\tfrac{\dd q\dd q'}{\l^d}\right)\de(x\!-\!y)\de(x'\!-\!y')
\\
-\left(\int_{\bR^{2d}}\l^2(\grad_x\!-\!\grad_{x'})a\left(\tfrac{x-q}{\sqrt\l},\tfrac{x'-q'}{\sqrt\l},\tfrac{y-q}{\sqrt\l},\tfrac{y'-q'}{\sqrt\l}\right)\tfrac{\dd q\dd q'}{\l^d}\right)\cdot(\grad_y\!-\!\grad_{y'})\de(x\!-\!y)\de(x'\!-\!y')
\\
-\left(\int_{\bR^{2d}}\l^2(\grad_y\!-\!\grad_{y'})a\left(\tfrac{x-q}{\sqrt\l},\tfrac{x'-q'}{\sqrt\l},\tfrac{y-q}{\sqrt\l},\tfrac{y'-q'}{\sqrt\l}\right)\tfrac{\dd q\dd q'}{\l^d}\right)\cdot(\grad_x\!-\!\grad_{x'})\de(x\!-\!y)\de(x'\!-\!y')
\\
=\l^2(\grad_x-\grad_{x'})\cdot(\grad_y-\grad_{y'})\de(x-y)\de(x'-y')
\\
+2\sqrt\l\Tr_{L^2(\bR^d,dz)}((R-R')\grad_z)\cdot (\grad_x\!-\!\grad_{x'})\de(x\!-\!y)\de(x'\!-\!y')
\\
-\l\Tr_{\fH\otimes\fH}((\grad_X-\grad_{X'})Q(\grad_X-\grad_{X'}))\de(x-y)\de(x'-y')&\,.
\ea
$$
The expression of the second term on the last right hand side comes from the identity
$$
\ba
\l^2\!\left(\int_{\bR^{2d}}(\grad_x\!-\!\grad_{x'})a\left(\tfrac{x-q}{\sqrt\l},\tfrac{x'-q'}{\sqrt\l},\tfrac{y-q}{\sqrt\l},\tfrac{y'-q'}{\sqrt\l}\right)\tfrac{\dd q\dd q'}{\l^d}\right)\cdot (\grad_y\!-\!\grad_{y'})\de(x\!-\!y)\de(x'\!-\!y')
\\
=\l^2\!\left(\int_{\bR^{2d}}(\grad_y\!-\!\grad_{y'})a\left(\tfrac{x-q}{\sqrt\l},\tfrac{x'-q'}{\sqrt\l},\tfrac{y-q}{\sqrt\l},\tfrac{y'-q'}{\sqrt\l}\right)\tfrac{\dd q\dd q'}{\l^d}\right)\cdot (\grad_x\!-\!\grad_{x'})\de(x\!-\!y)\de(x'\!-\!y')&\,,
\ea
$$
which holds since 
$$
\int_{\bR^{2d}}a\left(\tfrac{x-q}{\sqrt\l},\tfrac{x'-q'}{\sqrt\l},\tfrac{y-q}{\sqrt\l},\tfrac{y'-q'}{\sqrt\l}\right){\dd q\dd q'}
$$ 
depends on $x-y$ and $x'-y'$ only, and from the formula
$$
\ba
\int_{\bR^d\times\bR^d}(\grad_x\!-\!\grad_{x'})a\left(\tfrac{x-q}{\sqrt\l},\tfrac{x'-q'}{\sqrt\l},\tfrac{x-q}{\sqrt\l},\tfrac{x'-q'}{\sqrt\l}\right)\tfrac{\dd q\dd q'}{\l^d}
\\
=\l^{-1/2}\Tr_{\fH\otimes\fH}((\grad\otimes\rI-\rI\otimes\grad)Q)
\\
=\l^{-1/2}\Tr_{\fH}\grad(R-R')&\,.
\ea
$$
The expression of the third term on the last right hand side comes from the identity
$$
\ba
(\grad_x\!-\!\grad_{x'})\cdot(\grad_y\!-\!\grad_{y'})a\left(\tfrac{x-q}{\sqrt\l},\tfrac{x'-q'}{\sqrt\l},\tfrac{y-q}{\sqrt\l},\tfrac{y'-q'}{\sqrt\l}\right)\rstr_{x=y,x'=y'}
\\
 =\l^{-1} (\grad_X\!-\!\grad_{X'})\cdot(\grad_X\!-\!\grad_{X'}) a(X,X',X,X')\rstr_{X=\frac{x-q}{\sqrt\l},X'=\frac{x'-q'}{\sqrt\l}}&\,.
\ea
$$

Finally, the conclusion follows from observing that 
$$
(\grad_x-\grad_{x'})\cdot(\grad_y-\grad_{y'})\de(x-y)\de(x'-y')
$$
is the integral kernel (in the sense of distributions) of the unbounded operator
$$
-|\grad_{x}-\grad_{x'}|^2\,,
$$
while
$$
\ba
\Tr_{\fH\otimes\fH}(|X-X'|^2Q)-\Tr_{\fH\otimes\fH}((\grad_X-\grad_{X'})Q(\grad_X-\grad_{X'}))
\\
=\Tr_{\fH\otimes\fH}(Q^{1/2}C_1Q^{1/2})&\,.
\ea
$$
\end{proof}

\smallskip
The next lemma is the analogue of formula (54) in \cite{FGMouPaul}.

\begin{Lem}\lb{L-TrOpfR}
For each trace-class operator $K$ on $L^2(\bR^d)$ and each bounded continuous function $f$ on $\bR^d$, 
$$
\Tr(\Op^R_\l[f]^*K)=\int_{\bR^d\times\bR^d}\overline{f(q,p)}\tilde W^R_\l[K](q,p)\dd p\dd q\,.
$$
\end{Lem}

\begin{proof}[Proof of Lemma \ref{L-TrOpfR}]
By formula (\ref{TrQ*R}), one finds that
$$
\ba
\Tr(\Op^R_\l[f]^*K)=(2\pi\l)^d\int_{\bR^d\times\bR^d}\overline{W_\l[\,\Op^R_\l[f]\,](x,\xi)}W_\l[K](x,\xi)\dd x\dd\xi
\\
=\int_{\bR^d\times\bR^d}\overline{f\star W_\l[\,R\,](x,\xi)}W_\l[R](x,\xi)\dd x\dd\xi
\\
=\int_{\bR^d\times\bR^d}\overline{f(q,p)}\left(W_\l[K]\star W_\l[\,R\,]^*\right)(q,p)\dd p\dd q
\\
=\int_{\bR^d\times\bR^d}\overline{f(q,p)}\tilde W^R_\l[K](q,p)\dd p\dd q
\ea
$$
by definition of the generalized Husimi transform (see Definition \ref{D-GenHusimi}).
\end{proof}

\begin{proof}[Proof of Theorem \ref{T-LBound}]
By the positivity of the quantization $\Op^R_\l$, assuming that $f$ and $g$ are real-valued, continuous bounded functions on $\bR^d\times\bR^d$ satisfying
\be\lb{fgIneq}
f(q,p)+g(q',p')\le|q-q'|^2+|p-p'|^2
\ee
for all $p,p',q,q'\in\bR^d$, one has
$$
\ba
\Op^R_\l[f]\otimes\rI_\fH+\rI_\fH\otimes\Op^R_\l[g]
\\
=
\frac1{(2\pi\l)^{2d}}\int_{(\bR^d\times\bR^d)^2}(f(p,q)+g(p',q'))Q^\l_{q,q',p/\l,p'/\l}\dd q\dd q'\dd p\dd p'
\\
\le C_\l(x,x',\grad_x,\grad_{x'})+\l\Tr_{\fH\otimes\fH}(Q^{1/2}C_1Q^{1/2})\rI_{\fH\otimes\fH}
\\
+2\l^{3/2}\Tr_{L^2(\bR^d,dz)}(R-R')\grad_z)\cdot(\grad_x-\grad_{x'}))
\\
-2\sqrt\lambda\Tr_{L^2(\bR^d,dz)}((R-R')z)\cdot(x-x')
\ea
$$
for each $Q\in\cC(R,R')$.

For each $L\in\cC(K,K')$, one has
$$
\ba
\Tr_{\fH\otimes\fH}(L^{1/2}C_\l(x,x',\grad_x,\grad_{x'})L^{1/2})+\l\Tr_{\fH\otimes\fH}(Q^{1/2}C_1Q^{1/2})
\\
-2\sqrt\l\Tr_{L^2(\bR^d,dz)}((R-R')z)\cdot\Tr_{L^2(\bR^d,dy)}(y(K-K'))
\\
+2\l^{3/2}\Tr_{L^2(\bR^d,dz)}((R-R')\grad_z\cdot\Tr_{L^2(\bR^d,dy)}(\grad_y(K-K'))
\\
\ge\Tr_\fH(\Op^R_\l[f]K)+\Tr(\Op^K_\l[g]K')
\\
=\int_{\bR^d\times\bR^d}f(q,p)\tilde W^R_\l[K](q,p)\dd p\dd q+\int_{\bR^d\times\bR^d}g(q',p')\tilde W^R_\l[K'](q',p')\dd p'\dd q'&\,.
\ea
$$
Minimizing the left-hand side of this inequality as $L$ and $Q$ run through $\cC(K,K')$ and $\cC(R,R')$ respectively, one finds that
$$
\ba
\MK_\l(K,K')^2+\l\MK_1(R,R')
\\
-2\sqrt\l\Tr_{L^2(\bR^d,dz)}((R-R')z)\cdot\Tr_{L^2(\bR^d,dy)}(y(K-K'))
\\
+2\l^{3/2}\Tr_{L^2(\bR^d,dz)}((R-R')\grad_z\cdot\Tr_{L^2(\bR^d,dy)}(\grad_y(K-K'))
\\
\ge\int_{\bR^d\times\bR^d}f(q,p)\tilde W^R_\l[K](q,p)\dd p\dd qs+\int_{\bR^d\times\bR^d}g(q',p')\tilde W^R_\l[K'](q',p')\dd p'\dd q'
\ea
$$
for all real-valued, bounded continuous functions $f,g$ on $\bR^d\times\bR^d$ satisfying (\ref{fgIneq}). Maximizing the right-hand side of this inequality in $f,g$ and applying Kantorovich duality (see Theorem 1.3 in chapter 1 of \cite{VillaniTOT})
implies the announced lower bound.
\end{proof}


\section{Application to the Mean-Field Limit}


Let $V\equiv V(z)$ be a real-valued function defined on $\bR^d$ and satisfying
\be\lb{HypV}
V\in C^{1,1}(\bR^d)\,,\quad\grad V\in L^\infty(\bR^d)\,,\quad V(y)=V(-y)\hbox{ for all }y\in\bR^d\,.
\ee
Let $\rho_{\hbar,N}\equiv \rho_{\hbar,N}(t)\in\cD(L^2((\bR^d)^N))$ be the solution of the Cauchy problem for the $N$-body Heisenberg equation 
\be\lb{NHeis}
\left\{
\ba
{}&i\hbar\d_t\rho_{\hbar,N}=\sum_{j=1}^N[-\tfrac12\hbar^2\Dlt_{x_k},\rho_{\hbar,N}]+\frac1N\sum_{1\le j<k\le N}[V_{jk},\rho_{\hbar,N}]\,,
\\
&\rho_{\hbar,N}\rstr_{t=0}=\rho^{in}_{\hbar,N}\,,
\ea
\right.
\ee
where $\rho^{in}_{\hbar,N}\in\cD^2(L^2((\bR^d)^N))$ is a given density operator. We have denoted $V_{jk}$ the operator on $L^2((\bR^d)^N)$ defined by the formula
$$
(V_{jk}\psi_N)(x_1,\ldots,x_N):=V(x_j-x_k)\psi_N(x_1,\ldots,x_N)\,.
$$
On the other hand, let $\rho_{\hbar}\equiv\rho_{\hbar}(t)\in\cD(L^2(\bR^d))$ be the solution of the Hartree equation 
\be\lb{Hartree}
\left\{
\ba
{}&i\hbar\d_t\rho_{\hbar}=[-\tfrac12\hbar^2\Dlt_{x},\rho_{\hbar}]+[V_{\rho_{\hbar}},\rho_{\hbar}]\,,
\\
&\rho_{\hbar}\rstr_{t=0}=\rho_{\hbar}^{in}\,,
\ea
\right.
\ee
where $\rho^{in}_{\hbar}\in\cD(L^2(\bR^d))$ is a given density operator. The notation $V_{\rho_{\hbar}}$ designates the time-dependent, mean-field potential defined by $\rho_{\hbar}(t)$, i.e.
$$
V_{\rho_{\hbar}(t)}(x):=\Tr((\tau_xV)\rho_{\hbar}(t))\quad\hbox{ where }(\tau_xV)\psi(y):=V(y-x)\psi(y)\,.
$$
If $r_{\hbar}(t,x,y)$ is the integral kernel of the density operator $\rho_{\hbar}(t)$, the operator $V_{\rho_{\hbar}(t)}$ is the (time-dependent) multiplication operator on $L^2(\bR^d)$ by the function
$$
x\mapsto\int_{\bR^d}V(x-z)r_{\hbar}(t,z,z)dz\,.
$$
Denote by $\cD_s(L^2((\bR^d)^N))$ the set of symmetric density operators on $L^2((\bR^d)^N)$, i.e. the density operators whose integral kernel $r\equiv r(x_1,\ldots,x_N,y_1,\ldots,y_N)$ satisfy the condition
\be\lb{SymDens}
r(x_1,\ldots,x_N,y_1,\ldots,y_N)=r(x_{\si(1)},\ldots,x_{\si(N)},y_{\si(1)},\ldots,y_{\si(N)})
\ee
for all $\si\in\fS_N$ (the symmetric group on $\{1,\ldots,N\}$). In quantum mechanics, the density operator for a set of $N$ indistinguishable particles satisfies (\ref{SymDens}).

Theorem 2.4 in \cite{FGMouPaul} states that, for all $n=1,\ldots,N$ and all $\rho^{in}_{\hbar,N}\in\cD_s(L^2((\bR^d)^N))$, one has
$$
\frac1n\MK_{\hbar}(\rho_{\hbar}(t)^{\otimes n},\rho_{\hbar,N}^\mathbf{n}(t))^2\le\frac8N\|\grad V\|_{L^\infty}\frac{e^{\L t}-1}{\L}+\frac{e^{\L t}}N\MK_{\hbar}((\rho^{in}_{\hbar})^{\otimes N},\rho^{in}_{\hbar,N})^2
$$
for all $t\ge 0$, where $\L:=3+4\Lip(\grad V)^2$. We have denoted by $\rho_{\hbar,N}^\mathbf{n}(t)$ the $n$-body marginal density associated to $\rho_{\hbar,N}(t)$, i.e.  the density operator with integral kernel
$$
\ba
r^\mathbf{n}_{\hbar,N}(t,x_1,\ldots,x_n,y_1,\ldots,y_n)
\\
:=\int_{(\bR^d)^{N-n}}r_{\hbar,N}(t,x_1,\ldots,x_n,z_{n+1},\ldots,z_N,y_1,\ldots,y_n,z_{n+1},\ldots,z_N)dz_{n+1}\ldots dz_N
\ea
$$
for $n=1,\ldots,N-1$, where $r_{\hbar,N}$ is the integral kernel of $\rho^{in}_{\hbar,N}$. We also set 
$$
\rho^\mathbf{N}_{\hbar,N}(t):=\rho_{\hbar,N}(t)\,.
$$

For $\hbar>0$ fixed, the mean-field limit, i.e. the approximation of $\rho_{\hbar,N}^\mathbf{1}(t)$ by $\rho_{\hbar}(t)$ in the large $N$ limit, has been studied by several authors (see for instance 
\cite{Spohn1980,BGM,BEGMY,EY,RodSchl,Pickl}, and the bibliography in \cite{FGMouPaul} for a more complete list of references).

The question of obtaining a uniform as $\hbar\to 0$ rate of convergence for the mean-field limit reduces  therefore  to obtaining an upper bound for
$$
\frac1N\MK_{\hbar}((\rho^{in}_{\hbar})^{\otimes N},\rho^{in}_{\hbar,N})^2\,,
$$
and a lower bound for 
$$
\frac1n\MK_{\hbar}(\rho_{\hbar}(t)^{\otimes n},\rho_{\hbar,N}^\mathbf{n}(t))^2\,,
$$
in terms of quantities better understood, and in particular involving a true distance. 

Theorem \ref{T-LBound} above (a generalization of  Theorem 2.3 (2) in \cite{FGMouPaul}) provides a family of such  lower bounds. Specializing it to $R=R'$ (the extension to the general case is trivial) one obtains that,
for any $R'\in\cD^2(L^2(\bR^d))$,
$$
\ba
\frac1n\MKd(\tilde W^{R'^{\otimes n}}_{\hbar}[\,\rho_{\hbar}(t)^{\otimes n}\,],\tilde W^{R'^{\otimes n}}_{\hbar}[\,\rho_{\hbar,N}^\mathbf{n}(t)\,])^2
\\
\le\frac8N\|\grad V\|_{L^\infty}\frac{e^{\L t}-1}{\L}+\frac{e^{\L t}}N MK_{\hbar}((\rho^{in}_{\hbar})^{\otimes N},\rho^{in}_{\hbar,N})^2+\frac\hbar n MK_1(R'^{\otimes n},R'^{\otimes n})&\,.
\ea
$$

An upper bound for the second term on the right hand side of the inequality above is obtained by Theorem \ref{T-LBound} above (generalization of  Theorem 2.3 (1) in \cite{FGMouPaul}): one can take initial data which are generalized 
T\"oplitz operators
\be\lb{CondinT}
\rho^{in}_{\hbar}:=\Op^R_{\hbar}[\,(2\pi\hbar)^d\mu^{in}_{\hbar}\,]\hbox{ and }\rho^{in}_{\hbar,N}=\Op^{R^{\otimes N}}_{\hbar}[\,(2\pi\hbar)^{dN}\mu^{in}_{\hbar,N}\,]
\ee 
in the sense of Definition \ref{D-Opmu}, for any $R\in\cD^2(L^2(\bR^d))$ and $\mu^{in}_{\hbar}$, $\mu^{in}_{\hbar,N}$  Borel probability measures on $\bR^d\times\bR^d$, $(\bR^d\times\bR^d)^N$ respectively, assuming that 
$\mu^{in}_{\hbar,N}$ is symmetric --- in other words, $\mu^{in}_{\hbar,N}$ is invariant under all transformations of the form
$$
(p_1,\ldots,p_N,q_1,\ldots,q_N)\mapsto(p_{\si(1)},\ldots,p_{\si(N)},q_{\si(1)},\ldots,q_{\si(N)})
$$
for all permutation $\si\in\fS_N$. Then one finds that
$$
\ba
\frac1n\MKd(\tilde W^{R'^{\otimes n}}_{\hbar}[\,\rho_{\hbar}(t)^{\otimes n}\,],\tilde W^{R'^{\otimes n}}_{\hbar}[\,\rho_{\hbar,N}^\mathbf{n}(t)\,])^2
\\
\le\frac8N\|\grad V\|_{L^\infty}\frac{e^{\L t}-1}{\L}+\frac{e^{\L t}}N \MKd\left((\mu^{in}_{\hbar})^{\otimes N},\mu^{in}_{\hbar,N}\right)^2
\\
+\hbar\left(\frac1n\MK_1(R'^{\otimes n},R'^{\otimes n})^2+\frac{e^{\L t}}N\MK_1(R^{\otimes N},R^{\otimes N})^2\right)&\,.
\ea
$$
The last term on the right hand side of this inequality is mastered by the following observation.

\begin{Lem}
Let $R_1,R_2\in\cD^2(\fH)$. For each integer $n\ge 1$, one has
$$
\MK_1(R_1^{\otimes n},R_2^{\otimes n})^2\le n\MK_1(R_1,R_2)^2\,.
$$
\end{Lem}

\begin{proof}
Let $Q\in\cC(R_1,R_2)$; then $Q^{\otimes n}\in\cC(R_1^{\otimes n}, R_2^{\otimes n})$. Denoting $X_N:=(x_1,\ldots,x_n)$ and $Y_n:=(y_1,\ldots,y_n)$, one has
$$
C_1(X_n,Y_n,\grad_{X_n},\grad_{Y_n})=\sum_{k=1}^nC_1(x_k,y_k,\grad_{x_k},\grad_{y_k})\,.
$$
Hence, for all $Q\in\cC(R_1,R_2)$, one has
$$
\ba
\MK_1(R_1^{\otimes n},R_2^{\otimes n})^2\le\Tr_{(\fH\otimes\fH)^{\otimes n}}((Q^{\otimes n})^{1/2}C_1(X_n,Y_n,\grad_{X_n},\grad_{Y_n})(Q^{\otimes n})^{1/2})
\\
=\sum_{k=1}^n\Tr_{(\fH\otimes\fH)^{\otimes n}}((Q^{\otimes n})^{1/2}C_1(x_k,y_k,\grad_{x_k},\grad_{y_k})(Q^{\otimes n})^{1/2})
\\
=n\Tr_{\fH\otimes\fH}(Q^{1/2}C_1(x,y,\grad_{x},\grad_{y})Q^{1/2})
\ea
$$
and the announced inequality follows from minimizing the right hand side as $Q$ runs through $\cC(R_1,R_2)$.
\end{proof}

\smallskip
With this observation, one arrives at the following convergence rate estimate.

\begin{Thm}\lb{T-CvgRate}
Let $R,R'\in\cD^2(L^2(\bR^d))$, and let $V\in C^{1;1}(\bR^d)$ be an even, real-valued potential. Let $\mu^{in}_{\hbar}$ and $\mu^{in}_{\hbar,N}$ be Borel probability measures on $\bR^d\times\bR^d$ and on $(\bR^d\times\bR^d)^N$ respectively,
such that $\mu^{in}_{\hbar,N}$ is symmetric in its $N$ phase space variables.

For $\rho^{in}_{\hbar}\equiv\Op^R_{\hbar}[(2\pi\hbar)^d\mu^{in}_{\hbar}]\in\cD^2(L^2(\bR^d))$, denote by $\rho_{\hbar}\equiv\rho_{\hbar}(t)\cD(L^2(\bR^d))$ the solution of the Cauchy problem (\ref{Hartree}) for the Hartree equation. 

Let $N\ge 1$, and for each $\rho^{in}_{\hbar,N}
\equiv\Op^{R^{\otimes N}}_{\hbar}[(2\pi\hbar)^{dN}\mu^{in}_{\hbar,N}]\in\cD_s(L^2((\bR^d)^N))$, denote by $\rho_{\hbar,N}\equiv\rho_{\hbar,N}(t)\in\cD_s(L^2((\bR^d)^N))$ the solution of the Cauchy 
problem (\ref{NHeis}) for the $N$-particle Heisenberg equation. Then, for each $n=1,\ldots,N$, one has
\be\lb{CvgRateEst}
\ba
\frac1n\MKd(\tilde W^{R'^{\otimes n}}_{\hbar}[\,\rho_{\hbar}(t)^{\otimes n}\,],\tilde W^{R'^{\otimes n}}_{\hbar}[\,\rho_{\hbar,N}^\mathbf{n}(t)\,])^2
\\
\le\frac8N\|\grad V\|_{L^\infty}\frac{e^{\L t}-1}{\L}+\frac{e^{\L t}}N \MKd\left((\mu^{in}_{\hbar})^{\otimes N},\mu^{in}_{\hbar,N}\right)^2
\\
+\hbar\left(\MK_1(R',R')^2+e^{\L t}\MK_1(R,R)^2\right)&\,.
\ea
\ee
\end{Thm}

\smallskip
This result calls for some remarks on the choice of the density operators $R$ and $R'$, and on the initial data for (\ref{Hartree}) and (\ref{NHeis}).

\smallskip
In order to improve the convergence rate estimate in Theorem \ref{T-CvgRate}, one must choose the density operators $R'$ so as to minimize the third term on the right hand side of (\ref{CvgRateEst}).

\smallskip
For instance, assume that $R'$ satisfies the condition
$$
\MK_1(R',R')^2=2d=\min_{\rho\in\cD^2(L^2(\bR^d))}\MK_1(\rho,\rho)^2\,.
$$
This would be the case with $R'=|a\ra\la a|$, where $a$ is the Gaussian density (\ref{Gauss}).

\smallskip
Next, if $\mu^{in}_{\hbar,N}=(\mu^{in}_{\hbar})^{\otimes N}$, the second term on the right hand side of (\ref{CvgRateEst}) vanishes and, with $R$ and $R'$ chosen as above, one finds that
\be\lb{MFUnif}
\ba
\frac1n\MKd(\tilde W^{R'^{\otimes n}}_{\hbar}[\,\rho_{\hbar}(t)^{\otimes n}\,],\tilde W^{R'^{\otimes n}}_{\hbar}[\,\rho_{\hbar,N}^\mathbf{n}(t)\,])^2
\\
\le\frac8N\|\grad V\|_{L^\infty}\frac{e^{\L t}-1}{\L}+\hbar(2d+e^{\L t}\MK_1(R,R)^2)&\,.
\ea
\ee

Another possible choice is
$$
R=R'=|a\ra\la a|
$$
and
\be\lb{CondinBosea}
\rho^{in}_{\hbar}=|p,q,\l,a\ra\la  p,q,\l,a|\quad\hbox{ and }\rho^{in}_{\hbar,N}=|p,q,\l,a\ra\la p,q,\l,a|^{\otimes N}\,,
\ee
for all $a\in H^1(\bR^d)$ satisfying
$$
\int_{\bR^d}|a(y)|^2\dd y=1,\int_{\bR^d}|y|^2|a(y)|^2\dd y<\infty\,,
$$ 
with $|p,q,\l,a\ra$ defined as in \eqref{defcoh}. 

In general
$$
\MK_1(|a\ra\la a|\,\,|a\ra\la a|)^2>2d\hbar
$$
so that the third term on the right hand side of (\ref{CvgRateEst}) is not minimal  with this choice of density operators $R$ and $R'$. Yet this class of examples is important, since the $N$-body density operator above is of the form
$$
\rho^{in}_{\hbar,N}=|\Psi^{in}_{\hbar,N}\ra\la\Psi^{in}_{\hbar,N}|
$$
where
$$
\Psi^{in}_{\hbar,N}(x_1,\ldots,x_N)=\prod_{k=1}^N| p,q\ra(x_k)\,.
$$
In particular, this class of initial data is defined in terms of a symmetric $N$-particle wave-function, i.e.
$$
\Psi^{in}_{\hbar,N}(x_{\si(1)},\ldots,x_{\si(N)})=\Psi^{in}_{\hbar,N}(x_1,\ldots,x_N)\quad\hbox{ for all }\si\in\fS_N\,.
$$
The corresponding density matrix satisfies the symmetry relation
\be\lb{Bose}
r^{in}_{\hbar,N}(x_{\si(1)},\ldots,x_{\si(N)},y_{\tau(1)},\ldots,y_{\tau(N)})=r^{in}_{\hbar,N}(x_1,\ldots,x_N,y_1,\ldots,y_N)
\ee
for all (possibly different) $\si,\tau\in\fS_N$, where $r^{in}_{\hbar,N}$ is the integral kernel of $R^{in}_{\hbar,N}$. This symmetry condition is of course more stringent than (\ref{SymDens}), and expresses the fact that the $N$ particles 
under consideration are bosons. Note that any factorized bosonic state is the tensor power of a one particle pure state.

In other words, combining Theorem 2.4 in \cite{FGMouPaul} with Theorems \ref{T-UBound} and \ref{T-LBound} above allows us to consider a larger class of initial data for which a uniform as $\hbar\to 0$ convergence rate of the form 
(\ref{MFUnif}) holds true. In particular, one can choose in this way many different initial conditions satisfying the Bose symmetry condition (\ref{Bose}), which states as in (\ref{CondinT}) may fail to satisfy, unless $\mu^{in}=\de_{p,q}$. 
We refer to chapter IX in  \cite{LL3} for a more detailed discussion of Bose statistics. 


\section{How to Metrize the Set of Quantum Densities?}


We shall conclude this paper with a few remarks on the problem of metrizing the set of quantum densities. For sake of simplicity we will state the result in the standard Gaussian T\"oplitz quantization, but the same arguments are valid for  general density matrices  as defined in this article.
 
 For $R_1,R_2\in\cD(L^2(\bR^d))$, it is customary in quantum mechanics to measure the distance between $R_1$ and $R_2$ in terms of the trace-norm (see for instance \cite{Spohn1980,RodSchl}) --- sometimes also in terms of the 
 Hilbert-Schmidt norm \cite{RodSchl} or of the operator norm \cite{Pickl}. 
 
 More generally, one can think of measuring the distance between $R_1$ and $R_2$ in terms of the Schatten norms
 $$
 \|R_1-R_2\|_{\cL^p(\bR^d)}\,,\qquad\hbox{ for }1\le p\le\infty\,.
 $$
 In this section, we denote by $\cL(\fH)$ the algebra of bounded operators on the (separable) Hilbert space $\fH$, and by $\|T\|$ the operator norm of $T\in\cL(\fH)$. For $p\in[1,\infty)$, the Schatten class $\cL^p(\fH)$ is the two-sided ideal
 of $\cL(\fH)$ whose elements are the operators $T\in\cL(\fH)$ such that $(T^*T)^{p/2}$ is trace-class, and we denote the Schatten norm on $\cL^p(\fH)$ by
 $$
 \|T\|_{\cL^p(\fH)}:=\Tr((T^*T)^{p/2})^{1/p}\,.
 $$
 In particular, $\cL^1(\fH)$ is the set of trace-class operators on $\fH$ and $\|T\|_{\cL^1(\fH)}$ the trace-norm of $T\in\cL^1(\fH)$, while $\cL^2(\fH)$ is the set of Hilbert-Schmidt operators on $\fH$ and $\|T\|_{\cL^1(\fH)}$ the Hilbert-Schmidt
 norm of $T\in\cL^2(\fH)$. (For more details on Schatten classes with exponent $p\in(1,\infty)\setminus\{2\}$, see Example 2 in the Appendix to IX.4 on p. 41 in \cite{RS2}; the more classical cases $p=1$ and $p=2$ are discussed in section
 VI.6 of \cite{RS1}.)
 
 Consider the special case
 $$
 R_1=|p_1,q_1\ra\la p_1,q_1|\,,\qquad R_2=|p_2,q_2\ra\la p_2,q_2|\,,
 $$
 assuming that $(p_1,q_1)\not=(p_2,q_2)$. Here $|p,q\ra$ are the coherent states as defined in \eqref{defcoh} with $\l=\hbar$ and $a$ is the standard Gaussian $a$ as defined in \eqref{Gauss}. In that case, $R_1-R_2$ is a self-adjoint operator satisfying
 $$
 \Tr(R_1-R_2)=0\quad\hbox{ and }\quad\hbox{rank}(R_1-R_2)=2\,.
 $$
 Hence
 $$
 \|R_1-R_2\|_{\cL^p(L^2(\bR^d))}=2^{1/p}\|R_1-R_2\|_{\cL(L^2(\bR^d))}\,,\qquad 1\le p<\infty\,.
 $$
 In particular
 $$
\|R_1-R_2\|_{\cL^p(L^2(\bR^d))}=2^{\frac1p-\frac12}\|R_1-R_2\|_{\cL^2(L^2(\bR^d))}\,,
$$
and the Hilbert-Schmidt norm $\|R_1-R_2\|_{\cL^2(L^2(\bR^d))}$ can be computed explicitly as follows:
$$
\ba
\|R_1-R_2\|^2_{\cL^p(L^2(\bR^d))}&=\Tr(R_1^2+R_2^2-R_1R_2-R_2R_1)
\\
&=\Tr(R_1+R_2-2R_1R_2)
\\
&=2(1-\Tr(R_1R_2))
\\
&=2(1-|\la p_1,q_1|p_2,q_2\ra|^2)\,,
\ea
$$
so that
$$
\|R_1-R_2\|_{\cL^p(L^2(\bR^d))}=2^{1/p}\sqrt{1-|\la p_1,q_1|p_2,q_2\ra|^2}\,.
$$
In the case where $a$ is the Gaussian (\ref{Gauss}), one can compute explicitly
$$
|\la p_1,q_1|p_2,q_2\ra|^2=e^{-(|p_1-p_2|^2+|q_1-q_2|^2)/2\hbar}\,,
$$
(by using Theorem VI.23 in \cite{RS1}) and hence
$$
\|R_1-R_2\|_{\cL^p(L^2(\bR^d))}=2^{1/p}\sqrt{1-e^{-(|p_1-p_2|^2+|q_1-q_2|^2)/2\hbar}}
$$
In the semiclassical limit, i.e. for $\hbar\to 0$, one has
$$
\|R_1-R_2\|_{\cL^p(L^2(\bR^d))}\to 2^{1/p}\de_{(p_1,q_1),(p_2,q_2)}
$$
where $\delta$ is the Kronecker symbol (i.e. $\de_{x,y}=0$ if $x\not=y$ and $\de_{x,x}=1$). In other words, in the semiclassical limit, all the metrics between orthogonal projections on coherent states defined in terms of Schatten norms 
converge (up to some unessential normalizing factor) to the discrete metric, defining the (uninteresting) trivial topology on the phase space. 

Put in other words, one should think of the quantum densities $R_1$ and $R_2$ as being the quantum analogues of the Dirac probability measures $\de_{(p_1,q_1)}$ and $\de_{(p_2,q_2)}$ respectively, defined on the phase space 
$\bR^d\times\bR^d$, and
$$
\|R_1-R_2\|_{\cL^p(L^2(\bR^d))}\to 2^{\frac1p-1}\|\de_{(p_1,q_1)}-\de_{(p_2,q_2)}\|_{TV}\quad\hbox{ as }\hbar\to 0\,,
$$
where $\|m\|_{TV}$ denotes the total variation of the signed measure $m$. 

In the semiclassical limit, quantum particles become perfectly localized on trajectories in phase space. The elementary computation above shows that the Schatten norms cannot detect distances between phase space points of order 
larger than $O(\hbar^{1/2})$, and are therefore unfit for measuring distances between points on trajectories in phase space.

At variance with the Schatten norms, the pseudo-distance $\MK_{\hbar}$ behaves like the Euclidean distance in phase space in the semiclassical limit, i.e. for $\hbar\to 0$. In the special case considered above, one has indeed,
by Corollary \ref{C-Gauss}
\be\lb{MKhPS}
\MK_{\hbar}(R_1,R_2)^2=|p_1-p_2|^2+|q_1-q_2|^2+\hbar\MK_1(|a\ra\la a|,|a\ra\la a|)\,.
\ee

Although $\MK_{\hbar}$ is not a distance\footnote{Indeed, $\MK_{\hbar}(R,R)^2\ge 2d\hbar$ for all $R\in\cD^2(L^2(\bR^d))$, according to formula (14) in \cite{FGMouPaul}. Also, we do not know whether $\MK_{\hbar}$ satisfies the triangle
inequality.} on $\cD(L^2(\bR^d))$, we believe that the few remarks above are the best justification for using $\MK_{\hbar}$ as a means of metrizing $\cD(L^2(\bR^d))$ in the context of the semiclassical limit of quantum mechanics.




\begin{thebibliography}{99}


\bibitem{BEGMY}
Bardos, C., Erd\"os, L., Golse, F., Mauser, N., Yau, H.-T.:
Derivation of the Schr\"odinger-Poisson equation from the quantum $N$-body problem. 
C. R. Math. Acad. Sci. Paris \textbf{334} (2002), 515--520.

\bibitem{BGM}
Bardos, C., Golse, F., Mauser, N.:
Weak coupling limit of the $N$ particle Schr\"odinger equation.
Methods Appl. Anal. \textbf{7} (2000), 275--293.

\bibitem{LL4}
Berestetskii, V.B., Lifshitz, E.M., Pitaevskii, L.P.:
``Quantum Electrodynamics''. 2nd edition.
Pergamon Press Ltd, 1982.

\bibitem{CohTan}
Cohen-Tannoudji, C., Diu, B., Lalo\"e, F.:
``Quantum Mechanics. Vol. 1''.
J. Wiley, New York, 1991.

\bibitem{EY}
Erd\"os, L., Yau, H.-T.:
Derivation of the nonlinear Schr\"odinger equation from a many body Coulomb system. 
Adv. Theor. Math. Phys. \textbf{5} (2001), 1169--1205.

\bibitem{FGMouPaul}
Golse, F., Mouhot, C., Paul, T.:
On the Mean Field and Classical Limits of Quantum Mechanics.
Commun. Math. Phys. \textbf{343} (2016) 165--205

\bibitem{GMP} 
Grossmann, A., Morlet, J., Paul,T.:
Transforms associated to square integrable representations I.
J. Math. Physics \textbf{26} (1985) 2473--2479.

\bibitem{HaurayGomez}
Gomez, C., Hauray, M.:
Rigorous derivation of Lindblad equations from quantum jump processes in 1D.
{\tt arXiv:1603.07969 [math-ph]}

\bibitem{LL3}
Landau, L.D., Lifshitz, E.M.:
``Quantum Mechanics. Nonrelativistic Theory''. 3rd edition.
Pergamon Press Ltd, 1977.

\bibitem{LionsPaul}
Lions, P.-L.,Paul, T.:
Sur les mesures de Wigner.
Rev. Math. Iberoam. \textbf{9} (1993), 553--618.

\bibitem{Pickl}
Pickl, P:
A simple derivation of mean-field limits for quantum systems.
Lett. Math. Phys. \textbf{97} (2011), 151--164.

\bibitem{RS1}
Reed, M., Simon, B.:
``Methods of Modern Mathematical Physics. I: Functional Analysis''.
Academic Press, 1980.

\bibitem{RS2}
Reed, M., Simon, B.:
``Methods of Modern Mathematical Physics. II: Fourier Analysis, Self-Adjointness''.
Academic Press, 1975.

\bibitem{RodSchl}
Rodnianski, I., Schlein, B.:
Quantum fluctuations and rate of convergence towards mean-field dynamics. 
Commun. Math. Phys. \textbf{291} (2009), 31--61.

\bibitem{schro}   
Schr\"odinger, E: 
Der stetige \"Ubergang von der Mikro- zur Makromechanik
Naturwiss., \textbf{14} (1926), 664--666.

\bibitem{Spohn1980}
Spohn, H:
Kinetic equations from hamiltonian dynamics. 
Rev. Mod. Phys. \textbf{52} (1980), 600--640.

\bibitem{VillaniTOT}
Villani, C.:
``Topics in Optimal Transportation''.
American Math. Soc, Providence (RI), 2003.

\end{thebibliography}
\end{document}